\documentclass{amsart}
\usepackage{amsmath, amsthm, amssymb, amscd, mathrsfs, eucal, epsfig}
\usepackage{pinlabel}

\begin{document}

\font\bfit=cmbxti10

\newtheorem{theorem}{Theorem}[section]
\newtheorem{lemma}[theorem]{Lemma}
\newtheorem{sublemma}[theorem]{Sublemma}
\newtheorem{proposition}[theorem]{Proposition}
\newtheorem{corollary}[theorem]{Corollary}
\newtheorem{conjecture}[theorem]{Conjecture}
\newtheorem{question}[theorem]{Question}
\newtheorem{problem}[theorem]{Problem}
\newtheorem*{claim}{Claim}
\newtheorem*{criterion}{Criterion}
\newtheorem*{reducible_thm}{Reducible Theorem (5.1)}
\newtheorem*{lens_thm}{Lens Theorem (5.2)}
\newtheorem*{hyperbolic_thm}{Hyperbolic Theorem (5.9)}
\newtheorem*{small_SFS_thm}{Small SFS Theorem (5.10)}
\newtheorem*{torus_thm}{Toroidal Theorem (5.19)}

\theoremstyle{definition}
\newtheorem{definition}[theorem]{Definition}
\newtheorem{construction}[theorem]{Construction}
\newtheorem{notation}[theorem]{Notation}

\theoremstyle{remark}
\newtheorem{remark}[theorem]{Remark}
\newtheorem{example}[theorem]{Example}

\numberwithin{equation}{subsection}

\def\H{\mathbb H}
\def\Z{\mathbb Z}
\def\N{\mathbb N}
\def\R{\mathbb R}
\def\Q{\mathbb Q}
\def\D{\mathcal D}
\def\E{\mathcal E}
\def\RR{\mathcal R}
\def\B{{\mathcal B}}
\def\PP{{\mathbb P}}

\def\RP{\mathbb{RP}}
\def\P{\mathcal P}
\def\F{\mathcal F}
\def\T{\mathcal T}
\def\ds{\displaystyle}

\def\CAT{\textnormal{CAT}}
\def\cl{\textnormal{cl}}
\def\scl{\textnormal{scl}}
\def\homeo{\textnormal{Homeo}}
\def\rot{\textnormal{rot}}
\def\area{\textnormal{area}}
\def\inte{\textnormal{int}}		

\def\Id{\textnormal{Id}}
\def\SL{\textnormal{SL}}
\def\GL{\textnormal{GL}}
\def\Sp{\textnormal{Sp}}
\def\PSL{\textnormal{PSL}}
\def\length{\textnormal{length}}
\def\fill{\textnormal{fill}}
\def\rank{\textnormal{rank}}
\def\til{\widetilde}

\title{Knots with small rational genus}
\author{Danny Calegari}
\address{Department of Mathematics \\ Caltech \\
Pasadena CA, 91125}
\email{dannyc@its.caltech.edu}
\author{Cameron Gordon}
\address{Department of Mathematics \\ University of Texas \\
Austin TX, 78712}
\email{gordon@math.utexas.edu}

\date{26/12/2012, Version 0.18}

\begin{abstract}
If $K$ is a rationally null-homologous knot in a $3$-manifold $M$, the {\em rational genus} of $K$
is the infimum of $-\chi(S)/2p$ over all embedded orientable surfaces $S$ in the complement of $K$
whose boundary wraps $p$ times around $K$ for some $p$ (hereafter: $S$ is a {\em $p$-Seifert surface}
for $K$). Knots with very small rational genus can be constructed by ``generic'' Dehn filling, 
and are therefore extremely plentiful. In this paper we show that knots with
rational genus less than $1/402$ are all {\em geometric} --- i.e. they may be isotoped into a special
form with respect to the geometric decomposition of $M$ --- and give a complete classification. 
Our arguments are a
mixture of hyperbolic geometry, combinatorics, and a careful study of the interaction of small
$p$-Seifert surfaces with essential subsurfaces in $M$ of non-negative Euler characteristic.
\end{abstract}

\maketitle

\section{Introduction}

Let $K$ be an oriented knot in a $3$-manifold $M$. If $K$ is null-homologous, it
bounds an embedded oriented surface $S$, 
called a Seifert surface. The least 
genus of such a surface is called the {\em genus} of $K$, and is denoted
$g(K)$. More generally, define the {\em rational genus} of $K$, denoted $\|K\|$, 
to be the infimum of $-\chi(S)/2p$ over all embedded orientable surfaces $S$ in the complement 
of $K$ without disk or sphere components, whose boundary wraps
$p$ times around $K$ for some $p$ (a precise definition is given in \S~\ref{p_Seifert_section}).

Largely because of an approach to the Berge conjecture via Knot Floer Homology, 
there has been recent interest in the question of
finding knots in $3$-manifolds with the property that they are
the {\em unique} knot in their homology class with least rational genus.
Since Knot Floer Homology detects the Thurston norm, and therefore
the rational genus of a knot (see Ozsv\'ath--Szab\'o \cite{OS}) such
knots have the property that they are characterized by their Knot Floer
Homology, and one can study such knots and surgeries on them using the surgery exact sequence. 
For example, the unknot in $S^3$ is the unique knot of genus $0$,
and various families of so-called ``simple knots'' in Lens spaces are the
unique knots of minimal rational genus in their respective homology classes (see Baker \cite{Baker},
Hedden \cite{Hedden} or Rasmussen \cite{Rasmussen}). Other authors, e.g.\ \cite{Baker_Etnyre},
have studied {\em rational linking number}, and its relation to contact geometry.

One way to produce knots of small rational genus is by surgery. For example,
let $K'$ be a non-trivial knot in $S^3$, and let $M$ be the $3$-manifold obtained by $p/q$ surgery on $K'$.
Let $K$ in $M$ be the core of the surgery solid torus. Then $[K]$ has order $p$ in
$H_1(M)$, and $\|K\| \le g(K')/p - 1/2p$ where $g(K')$ denotes the (ordinary) genus of $K'$, 
since the boundary of a Seifert surface for $K'$ wraps $p$ times around $K$ in the surgered manifold. 
For more detailed examples, see \S~\ref{example_subsection}.

\smallskip

The purpose of this paper is firstly to initiate a systematic study of rational genus and some of
its properties, and secondly to demonstrate that there is a universal positive constant such 
that knots in $3$-manifolds with rational genus bounded above by this constant can be 
completely classified. The precise statement of this classification, and the best estimate 
for the relevant constant, falls into several cases depending on the geometric decomposition of $M$.

\smallskip

In the generic (i.e.\ hyperbolic) case, the strategy is to deduce information about $K$ in two steps:
$$L^1\text{-homology} \longrightarrow \text{homotopy} \longrightarrow \text{isotopy}$$
An estimate for the rational genus of $K$ is really an estimate of the $L^1$-norm
of a certain relative homology class; such an estimate can be reinterpreted dually in terms of
bounded cohomology. Low-dimensional bounded cohomology in hyperbolic manifolds is
related to geometry (at the level of $\pi_1$) by the methods of Calegari \cite{Calegari_stable}.
Homotopy information is promoted to isotopy information by drilling and filling,
using uniform geometric estimates due to Hodgson-Kerckhoff \cite{Hodgson_Kerckhoff},
and Gauss-Bonnet. One interesting technical aspect of the argument is that it involves
finding a $\CAT(-1)$ representative of a surface in the complement of the cone locus
of a hyperbolic cone manifold. Such a surface can be found either by the PL wrapping technique of
Soma \cite{Soma}, or the shrinkwrapping technique of Calegari-Gabai \cite{Calegari_Gabai}.

\smallskip

The organization of the paper is as follows. \S~\ref{p_Seifert_section} introduces definitions,
proves some basic lemmas, and ends with several subsections enumerating examples. \S~\ref{graph_section}
introduces and develops the tools that power our combinatorial arguments (which apply especially
to knots with non-hyperbolic complements). \S~\ref{hyperbolic_knots_section} treats knots with hyperbolic
complements, and the arguments are more geometric and analytic. Finally, \S~\ref{general_knots_section}
assembles all this material, and contains the main theorems and their proofs.

\subsection{Statement of results}

The classification of knots with (sufficiently) small rational genus falls into
several cases, depending on the prime and geometric decomposition of $M$. These theorems are
proved in \S~\ref{general_knots_section}, and reproduced here for the convenience of the reader.

\begin{reducible_thm}
Let $K$ be a knot in a reducible manifold $M$. 
Then either 
\begin{enumerate}
\item{$\|K\| \ge 1/12$; or}
\item{there is a decomposition $M = M' \# M''$, $K \subset M'$ and either
\begin{enumerate}
\item{$M'$ is irreducible, or}
\item{$(M',K) = (\RP^3,\RP^1)\#(\RP^3,\RP^1)$}
\end{enumerate}
}
\end{enumerate}
\end{reducible_thm}

\begin{lens_thm}
Let $K$ be a knot in a lens space $M$. 
Then either 
\begin{enumerate}
\item{$\|K\| \ge 1/24$; or}
\item{$K$ lies on a Heegaard torus in $M$; or}
\item{$M$ is of the form $L(4k,2k-1)$ and $K$ lies on a Klein bottle in $M$ as a non-separating
orientation-preserving curve.}
\end{enumerate}
\end{lens_thm}

\begin{hyperbolic_thm}
Let $K$ be a knot in a closed hyperbolic $3$-manifold $M$.
Then either 
\begin{enumerate}
\item{$\|K\| \ge 1/402$; or}
\item{$K$ is trivial; or} 
\item{$K$ is isotopic to a cable of the core of a Margulis tube.}
\end{enumerate}
\end{hyperbolic_thm}

\begin{small_SFS_thm}
Let $M$ be an atoroidal Seifert fiber space over $S^2$ with three 
exceptional fibers and let $K$ be a knot in $M$. 
Then either
\begin{enumerate}
\item{$\|K\| \ge 1/402$; or}
\item{$K$ is trivial; or}
\item{$K$ is a cable of an exceptional Seifert fiber of $M$; or}
\item{$M$ is a prism manifold and $K$ is a fiber in the Seifert fiber structure of $M$ over $\RP^2$
with at most one exceptional fiber.}
\end{enumerate}
\end{small_SFS_thm}

\begin{torus_thm}
Let $M$ be a closed, irreducible, toroidal 3-manifold, and let $K$ be a 
knot in $M$. 
Then either 
\begin{enumerate}
\item{$\|K\| \ge 1/402$; or}
\item{$K$ is trivial; or}
\item{$K$ is contained in a hyperbolic piece $N$ of the 
JSJ decomposition of $M$ and is isotopic either to a cable of a core of a 
Margulis tube or into a component of $\partial N$; or}
\item{$K$ is contained in a Seifert fiber piece $N$ of the 
JSJ decomposition of $M$ and either
\begin{itemize}
\item[(A)]{$K$ is isotopic to an ordinary fiber or a cable of an exceptional fiber or into $\partial N$, or}
\item[(B)]{$N$ contains a copy $Q$ of the twisted $S^1$ bundle over the M\"obius band and $K$ is contained
in $Q$ as a fiber in this bundle structure;}
\end{itemize}
or}
\item{$M$ is a $T^2$-bundle over $S^1$ with Anosov monodromy and $K$ is contained in a fiber.}
\end{enumerate}
\end{torus_thm}

The subsections \S~\ref{example_subsection} and \S~\ref{torus_bundles_subsection}
discuss constructions giving rise to eight families of examples of knots with arbitrarily small
rational genus, illustrating that all the possibilities listed in the classification theorems 
really do occur.

\subsection{Acknowledgements}

The first author would like to thank Matthew Hedden and Jake Rasmussen for interesting
and stimulating talks they gave at Caltech in 2007, which were the inspiration for this
paper. He would also like to thank Marty Scharlemann and Yoav Rieck for useful conversations
about thin position. The second author would like to thank Constance Leidy and Peter Oszv\'ath for
useful comments. Danny Calegari was partially supported by NSF grant DMS 0707130.

\section{$p$-Seifert surfaces}\label{p_Seifert_section}

This section standardizes definitions, proves some basic facts about rational genus, and
describes how to construct examples of knots with small rational genus, illustrating the
significance of the cases we enumerate in our classification theorems.

\subsection{Definitions}

We formalize definitions in this section. Throughout, all $3$-manifolds considered will be compact, connected and
orientable. A {\em knot} $K$ in a $3$-manifold $M$ is a
tamely embedded $S^1$. If $K$ is null-homologous in $M$,
a {\em Seifert surface} for $K$ is a connected embedded two-sided surface
$S$ in $M$ with $\partial S = K$. The {\em genus} of $K$ is the least genus of a Seifert
surface.

This can be generalized as follows. By analogy with the Thurston norm on $H_2(M)$, we adopt the following notation:

\begin{notation}
If $S$ is a compact, orientable connected surface, define $\chi^-(S):= \min(0,\chi(S))$. If $S$ is a compact, orientable
surface with components $S_i$, define $\chi^-(S) = \sum_i \chi^-(S_i)$. Denote $\eta(S) = -\chi^-(S)/2$.
\end{notation}

\begin{remark}
The normalizing factor of $2$ in the denominator of $\eta$ reflects the fact that Euler characteristic is ``almost'' $-2$ 
times genus for a surface with a bounded number of boundary components.
\end{remark}

\begin{definition}
Let $K$ be a knot in a $3$-manifold $M$, with regular neighborhood $N(K)$. If $p$ is a positive integer, a
{\em $p$-Seifert surface} for $K$ is a compact, oriented surface $S$ embedded in $M - \inte N(K)$ such that
$S \cap \partial N(K) = \partial S$, and $[\partial S] = p[K] \in H_1(N(K))$ (for some choice of orientation on $K$).
In this case we define the {\em norm}, or {\em rational genus} of $K$ by
$$\|K\| = \inf_S \eta(S)/p$$
where the infimum is taken over all $p$ and all $p$-Seifert surfaces $S$ for $K$.
\end{definition}

A $p$-Seifert surface $S$ for $K$ can be extended into $N(K)$ to give a map $S \to M$ which is an
embedding on $\inte S$ and whose restriction $\partial S \to K$ is a (possibly disconnected)
covering map of degree $p$. We will regard a $p$-Seifert surface for $K$ as a singular surface in $M$
in this way in \S~\ref{hyperbolic_knots_section}. 

\begin{remark}
A knot $K$ has a $p$-Seifert surface for some $p$ if and only if $[K]$ has finite order in $H_1(M)$.
\end{remark}

\begin{definition}\label{good_definition}
A $p$-Seifert surface is {\em good} if it satisfies the following properties:
\begin{enumerate}
\item{$S$ is connected}
\item{$S$ is incompressible in $M - \inte N(K)$}
\item{$\partial S$ consists of $q$ parallel, coherently oriented copies of an essential simple closed curve $\rho$
on $\partial N(K)$ that represents $r$ times a generator of $H_1(N(K))$, where $qr=p$}
\end{enumerate}
\end{definition}

\begin{lemma}\label{p_surface_is_good}
Let $S$ be a $p$-Seifert surface for $K$. Then there is a good $p'$-Seifert surface $S'$ for $K$ satisfying
$$\eta(S)/p \ge \eta(S')/p'$$
\end{lemma}
\begin{proof}
Boundary components of $S$ that are inessential in $\partial N(K)$ may be 
capped off with disks, and closed components of $S$ may be discarded; 
neither of these operations increases $\eta$ or changes $p$. 
If some component of $S$ is a disk $D$, then $D$ is a good $p'$-Seifert 
surface for $K$ for some $p'$, and $\eta (D)$ and $\|K\|$ are both zero.
Hence we may assume that every component $S_i$ of $S$ satisfies 
$\chi (S_i) = \chi^- (S_i)$, and has non-empty boundary, each component of 
which is an essential curve in $\partial N(K)$. 

Since a $p$-Seifert surface is embedded, the components of $\partial S$ are all parallel in $\partial N(K)$, and are
therefore all isotopic (though {\it a priori} they might have opposite orientations).
Let $S$ be a $p$-Seifert surface for $K$, and suppose there are a pair of adjacent components of $\partial S$ on
$\partial N(K)$ that are oppositely oriented. Tubing this pair of components does not affect $\eta(S)$ or $p$,
so without loss of generality we may assume that all components are coherently oriented.

If $S$ is compressible in $M-\inte N(K)$ then compressing it along a disk gives a surface $S'$ with
$\partial S' = \partial S$ and $\eta(S') \le \eta(S)$. So we may assume that $S$ is incompressible.

It remains to show that we can take $S$ to be connected. Suppose $S$ is the disjoint union of $S_1$ and $S_2$,
where $\partial S_i$ consists of $q_i$ copies of $\rho$, and each $q_i > 0$. Then we estimate
$$\frac {\eta(S)} {p} = \frac {\eta(S)} {qr} = \frac {\eta(S_1) + \eta(S_2)} {(q_1 + q_2)r} \ge \min \left\{
\frac {\eta(S_1)} {q_1r}, \frac {\eta(S_2)} {q_2r} \right\}$$
\end{proof}

It will be important in the sequel to consider surfaces in manifolds $M$ that meet $\partial M$. 

\begin{definition}
A {\em relative $p$-Seifert surface} $F$ for a knot $K$ in $M$ is an oriented surface, properly embedded in
$M - \inte N(K)$ such that $[\partial F \cap \partial N(K)] = p[K] \in H_1(N(K))$.
\end{definition}

The definition of {\em good} extends to relative $p$-Seifert surfaces, and Lemma~\ref{p_surface_is_good}
generalizes to such surfaces as well and with the same proof, so in the sequel we assume all our $p$-Seifert surfaces,
relative or otherwise, are good.

\begin{notation}
In the sequel, we denote $X:= M - \inte N(K)$.
\end{notation}

\subsection{Thurston norm}

A basic reference for this section is Thurston's paper \cite{Thurston_norm}.

If $K$ is a knot in a closed $3$-manifold $M$, then $X$ is
a compact $3$-manifold with torus boundary, and as is well-known, the kernel of the
inclusion map $$H_1(\partial X;\Q) \to H_1(X;\Q)$$ is $1$-dimensional, and denoted $L$. Consequently,
the kernel of $$H_1(\partial X;\Z) \to H_1(X;\Z)$$ is isomorphic to $\Z$, and is generated by a single
element $m[\rho]$, where $[\rho]$ is primitive in $H_1(\partial X;\Z)$, and $[\rho]$ is represented by a simple loop
$\rho \subset \partial X$ which represents $r$ times a generator of $H_1(N(K))$ (this is the same $\rho$ as before).

Let $\partial^{-1}L$ denote the subspace of $H_2(X,\partial X;\Q)$ that is the preimage of $L$ under the connecting
homomorphism in rational homology. If $S$ is a $p$-Seifert surface for $K$, then $p = qr$ where $m|q$, and 
$[S] \in \partial^{-1}L$. Consider the affine rational subspace $\partial^{-1}[\rho]/r \subset \partial^{-1}L$.
The multiple $[S]/p \in \partial^{-1}[\rho]/r$, and $\eta(S)/p = \|S\|_T/2p$ where $\|\cdot\|_T$ denotes the Thurston
norm of a surface. Hence, by the definition of rational genus and of Thurston norm, we obtain the following formula:

\begin{lemma}\label{rational_genus_is_Thurston_norm}
There is an equality
$$\|K\| = \inf_{A \in \partial^{-1}[\rho]/r}\|A\|_T/2$$
where $\|A\|_T$ denotes the Thurston norm of the (rational) homology class $A$.
\end{lemma}

Since $\|\cdot\|_T$ is a convex, non-negative, piecewise rational linear function on $H_2(X,\partial X;\R)$, 
the infimum of $\|\cdot\|_T/2$ is {\em achieved} (on some rational subpolyhedron)
on the rational affine subspace $\partial^{-1}[\rho]/r$. 
Consequently, we have:

\begin{proposition}\label{rational_genus_is_achieved}
Let $K$ be a knot. The rational genus $\|K\|$ is equal to $\eta(S)/p$ for some $p$-Seifert surface $S$ and some $p$.
Therefore $\|K\|$ is rational. Moreover, there is an algorithm to find $S$ and compute $\|K\|$.
\end{proposition}
\begin{proof}
There is a (straightforward) algorithm to compute the Thurston norm, described in \cite{Thurston_norm}, and to find a 
norm-minimizing surface in any integral class (note that such a surface can be taken to be normal relative to a
fixed triangulation, and therefore may be found by linear programming in normal surface space).
\end{proof}

Although Proposition~\ref{rational_genus_is_achieved} is included for completeness, 
it is not essential for the remainder of the paper, and it is generally good enough in 
the sequel to work with a $p$-Seifert surface that comes close to realizing $\|K\|$.

\medskip

The first thing one wants to know about an invariant is when it vanishes. 
\begin{theorem}\label{vanishing_theorem}
Let $K$ be a knot in a 3-manifold $M$. 
Then $\|K\| =0$ if and only if either 
\begin{enumerate}
\item{$K$ bounds a disk in $M$; or}
\item{$K$ is the core of a genus 1 Heegaard splitting of a lens
space summand of $M$; or}
\item{$K$ is the fiber of multiplicity $r$ in a Seifert fiber 
subspace of $M$ whose base orbifold is a M\"obius band with one orbifold
point of order $r\ge 1$.}
\end{enumerate}
\end{theorem}

\begin{proof} 
Suppose $\|K\| =0$. 
Then by Lemma~\ref{p_surface_is_good} and Proposition~\ref{rational_genus_is_achieved},
$K$ has a good $p$-Seifert surface with $\chi^-(S)=0$, i.e.\ $S$ is a disk or annulus.

First assume $S$ is a disk. 
If $p=1$ we have conclusion~(1). 
If $p>1$ then a regular neighborhood of $N(K)\cup S$ is a punctured 
lens space with fundamental group $\Z/p\Z$, with $K$ as a core of 
a genus~1 Heegaard splitting.

If $S$ is an annulus, note that both boundary components of $S$ wrap with 
the same orientation $r$ times around $K$. 
A regular neighborhood of $N(K)\cup S$ evidently has the desired structure.
\end{proof}

\begin{remark}
Since $M$ is orientable, the total space of the Seifert fiber subspace in 
bullet~(3) of Theorem~\ref{vanishing_theorem} is orientable; 
i.e.\ it is a twisted $S^1$ bundle over an orbifold 
M\"obius band. 
There is no suggestion that it is essential in $M$. The situation described in~(3) arises
in case (2)(b) of Theorem~\ref{reducible_theorem}, case (3) of Theorem~\ref{lens_theorem}, case (4) of
Theorem~\ref{small_sfs_theorem}, and case (4)(B) of Theorem~\ref{toroidal_theorem}.
\end{remark}

Under suitable homological conditions on $M$ and $K$, the analysis simplifies considerably:

\begin{lemma}\label{homology_lemma}
Let $K$ be a knot in a $\mathbb{Q}$-homology $3$-sphere and let $S$ be a connected 
$p$-Seifert surface for $K$. Then $p$ is the order of $[K]$ in $H_1(M)$.
\end{lemma}

\begin{proof} 
Since $M$ is a $\mathbb{Q}$-homology sphere, $H_2(X)=0$ and hence the boundary map
$H_2(X,\partial X) \to H_1(\partial X)$ is injective, with image
$\ker (H_1(\partial X) \to H_1(X) )\cong \mathbb{Z}$. Let $x$ be a generator
of $H_2(X,\partial X)$. Then the image of $x$ under the composition 
$$H_2(X,\partial X) \to H_1(\partial X) \to H_1(N(K)) \cong \mathbb{Z}$$
is $p_0$, say, the order of $[K]$ in $H_1(M)$. Let $S$ be a connected $p$-Seifert
surface for $K$. Then $[S]=kx$ and $p=kp_0$ for some $k>0$. But since $S$ is
connected, $k=1$ by Lemma~1 of \cite{Thurston_norm}.
\end{proof}

If $K$ is a knot in a homology $3$-sphere, the (ordinary) {\em genus of $K$}, denoted $g(K)$, is
the minimal genus of any Seifert surface for $K$. Lemma~\ref{homology_lemma} reduces
the study of rational genus to that of the usual genus in homology spheres:

\begin{corollary}\label{homology_corollary}
If $K$ is a knot in a homology $3$-sphere then 
$$\|K\| = \begin{cases} 
0\ ,&\text{ if $g(K)=0$;}\\
\noalign{\vskip6pt}
g(K) -1/2\ ,&\text{ if $g(K)>0$.}
\end{cases}$$
\end{corollary}

\begin{proof} 
By Lemma~\ref{homology_lemma}, 
$\|K\| = \eta (S)$, where $S$ is a minimal genus Seifert surface for $K$. 
If $S$ is a disk, $\|K\| =g(K)=0$. Otherwise, $\eta (S) = (2g(S)-1)/2 = g(S) - 1/2$.
\end{proof}

The following lemma will allow us to construct knots in $3$-manifolds 
with arbitrarily small (non-zero) rational genus. 

\begin{lemma}\label{surgery_lemma} 
Let $K'$ be a knot in a homology 3-sphere $M'$.
Let $M$ be the manifold obtained by $m/n$-Dehn surgery on $K'$, where
$m>0$, and let $K\subset M$ be the core of the surgery solid torus. 
Then $\|K\| = \|K'\|/m$.
\end{lemma} 

\begin{proof} 
Note that $[K]$ has order $m$ in $H_1(M) = \Z/m\Z$. 

Let $S$ be a good $p$-Seifert surface for $K$ in $M$ such that 
$\|K\| = \eta (S)/p$. The restriction of $S$ to $M-K = M'-K'$ extends to a good $p'$-Seifert surface
for $K'$ (which by abuse of notation we call $S$) where $p'=p/m$. 
By the proof of Lemma~\ref{homology_lemma}, $p'=1$ and $S$ is a Seifert 
surface for $K'$ in $M'$. Conversely a Seifert surface for $K'$ can also be thought of as a
good $m$-Seifert surface for $K$. Therefore $\|K'\| =\eta (S)$ and $p=m$. 
Hence $\|K\| = \|K'\|/m$.
\end{proof}

\subsection{Connect sums}

We now examine the behavior of rational genus under connected sum. 
In this context it is convenient to say that a knot $K$ in $M$ is 
{\em $p$-trivial\/} if it has a $p$-Seifert surface that is a disk.
If $p=1$ then $K$ bounds a disk, and is {\em trivial\/}. 
If $p>1$ then $K$ is the core of a genus~1 Heegaard splitting of a lens space
summand of $M$ with fundamental group $\Z/p\Z$.

\begin{remark}\label{remark_200}
If $K$ is $p$-trivial for some $p$ then clearly $\partial N(K)$ is 
compressible in $X$. 
Conversely, a compressing disk for $\partial N(K)$ in $X$ is either a 
$p$-Seifert surface for $K$, for some $p\ge1$, or has boundary a meridian 
of $K$, in which case $K$ is isotopic to $S^1\times\{\text{point}\}$ in some 
$S^1\times S^2$ summand of $M$.
So for rationally null-homologous knots, being $p$-trivial for some $p$ is 
equivalent to $\partial N(K)$ being compressible in $X$.
\end{remark}

\begin{theorem}\label{connect_sum_theorem} 
Let $K_1$ and $K_2$ be knots in $3$-manifolds $M_1$ and $M_2$. 
Then 

\begin{enumerate}
\item{if $K_1$ is $p_1$-trivial and $K_2$ is trivial then 
$K_1\,\#\, K_2$ is $p_1$-trivial;}
\item{
\begin{equation*}
\|K_1\,\#\, K_2\| = 
\begin{cases} 
\ds \|K_1\| + \|K_2\| + \frac12 \quad
\text{if $K_1$ and $K_2$ are not $p$-trivial for any $p$};\\
\noalign{\vskip6pt}
\ds \|K_1\| +\frac12 - \frac1{2p_2} \quad
\text{if $K_2$ is $p_2$-trivial, $K_1$ is not $p$-trivial for any $p$};\\
\noalign{\vskip6pt}
\ds \frac12 - \frac{(p_1+p_2)}{2p_1p_2} \quad
\text{if $K_i$ is $p_i$-trivial, $p_i\ge 2$, $i=1,2$}.
\end{cases}
\end{equation*}}
\end{enumerate}
\end{theorem}

\begin{remark}
The first case in bullet~(2) 
says that for knots that are not $p$-trivial for any $p$, the quantity
$\|K\| +\frac12$ is additive under connected sum, 
and is the analog of the additivity of genus for knots in $S^3$; 
see Corollary~\ref{homology_corollary}. 
\end{remark}

\begin{remark}
Theorem~\ref{connect_sum_theorem} has an analog in the theory of stable commutator length (see
Definition~\ref{scl_definition});
compare with the Product formula (Theorem~2.93) from \cite{Calegari_scl}.
\end{remark}

\begin{remark}
Note that $K$ is 2-trivial if and only if $K$ is contained in an
$\RP^3$ summand as $\RP^1$.
Also, it follows from Theorem~\ref{connect_sum_theorem} that if $K= K_1\,\#\,K_2$ with
$K_1$ and $K_2$ non-trivial, then $\|K\| =0$ if and only if $K_1$ and
$K_2$ are 2-trivial, i.e. $K$ is contained in an $\RP^3\,\#\,\RP^3$
summand as $\RP^1\,\#\,\RP^1$.
This is a special case of Theorem~\ref{vanishing_theorem} (3), with $r=1$.
\end{remark}

\begin{proof}[Proof of Theorem~\ref{connect_sum_theorem}] 
Let $S_i \subset M_i$ be a good $p_i$-Seifert surface for $K_i$ with 
$\|K_i\| = \eta (S_i)/p_i$, $i=1,2$. 
Then we may construct a $p_1p_2$-Seifert surface $S$ for $K= K_1\,\#\,K_2$ 
in $M_1\,\#\,M_2$ by taking $p_2$ copies of $S_1$ and $p_1$ copies of $S_2$ 
and joining them along $p_1p_2$ arcs. 
We have 
\begin{equation}\label{eq1}
\chi (S) = p_2 \chi (S_1) + p_1 \chi (S_2) - p_1 p_2
\end{equation}
Conversely, let $S$ be a good $p$-Seifert surface for $K=K_1\,\#\,K_2$ with 
$\|K\| = \eta (S) /p$. 
Suppose $\partial S$ has $q$ components, each having intersection number $r$ 
with a meridian of $K$. 
Let $A$ be the annulus in $X = M_1\,\#\, M_2 - \inte N(K)$ that 
realizes the connected sum decomposition. 
By an isotopy of $S$ we may assume that each component of $\partial S$ 
meets each component of $\partial A$ in $r$ points, and that $S\cap A$ 
is a disjoint union of arcs and simple closed curves. 
Since the boundary components of $S$ are oriented coherently on 
$\partial N(K)$, each arc must have one endpoint on each component of 
$\partial A$. 
It follows that any simple closed curve of intersection is inessential in $A$, 
and therefore in $S$, and so these can be removed by performing surgery 
on $S$ and discarding the resulting 2-spheres. 
Hence we may assume that $S\cap A$ consists of $qr=p$ essential arcs in $A$. 
Cutting $S$ along these arcs gives $p$-Seifert surfaces $S_1,S_2$ for 
$K_1,K_2$ in $M_1,M_2$. 
Note that 
\begin{equation}\label{eq2}
\chi (S) = \chi (S_1) + \chi (S_2) - p
\end{equation}

To prove part (1), note that if $S_1$ and $S_2$ are disks and $p_2=1$
then the $p_1$-Seifert surface for $K_1\,\#\, K_2$ in \eqref{eq1} is a disk.

To prove (2), 
first suppose that  $K_1$ and $K_2$ are not $p$-trivial (for any $p$). 
Then the surfaces $S_1$ and $S_2$ in \eqref{eq1} and \eqref{eq2}
have no disk components, and from \eqref{eq1}
we get 
$$\eta (S) = p_2 \eta (S_1) + p_1\eta (S_2) + p_1 p_2/2\ ,$$
and hence 
\begin{equation}\label{eq3}
\begin{split}
\|K\| \le \eta (S) /p_1 p_2 
&= \eta (S_1)/p_1 + \eta (S_2)/p_2 + \frac12 \\
&= \|K_1\| + \|K_2\| + \frac12\ .
\end{split}
\end{equation}
Similarly, \eqref{eq2} gives 
$$\eta (S) = \eta (S_1) + \eta (S_2) + p/2\ ,$$
and hence 
\begin{equation}\label{eq4} 
\begin{split}
\|K\| = \eta (S)/p 
&= \eta (S_1)/p + \eta (S_2)/p + \frac12\\
&\ge \|K_1 \| + \|K_2\| + \frac12\ .
\end{split}
\end{equation}
Together, \eqref{eq3} and \eqref{eq4} 
give the first assertion in  part~(2) of the theorem. 

Second, suppose $K_2$ is $p_2$-trivial and $K_1$ is not $p$-trivial 
for any $p$.
Then in \eqref{eq1}  $S_2$ is a disk while $\chi (S_1)\le 0$.
Hence $\chi (S)\le 0$, and we get 
\begin{equation}\label{eq5} 
\begin{split}
\|K\| \le \eta (S) /p_1 p_2
& = \eta (S_1)/p_1 - \frac1{2p_2} +\frac12\\
& = \|K_1 \| - \frac1{2p_2} + \frac12\ .
\end{split}
\end{equation}
In \eqref{eq2}, since $K_1$ is not $p$-trivial for any $p$, no 
component of $S_1$ is a disk.
Hence $\eta (S_1) = - \chi (S_1)/2$, $\chi(S_1)\le 0$, and $\chi (S) \le0$.
Let $S_2$ have $d_2$ disk components. 
Then $\chi (S_2) \le d_2$ and 
\begin{equation}\label{eq6}
d_2 p_2 \le p\ .
\end{equation}

Now 
\begin{equation}\label{eq7}
\begin{split} 
\|K\| = \eta (S)/p & = - \chi(S)/2p\\
&\ge -\chi (S_1)/2p - \frac{d_2}{2p} + \frac12\\
&\ge \|K_1\| - \frac{d_2}{2p} + \frac12\ .
\end{split}
\end{equation}
Comparing \eqref{eq7} with \eqref{eq5} we get 
$$\frac{d_2}{2p} \ge \frac1{2p_2}\ ,\qquad \text{i.e. }\ d_2p_2\ge p\ .$$
By \eqref{eq6}, this gives $d_2p_2 =p$ and \eqref{eq5} is an equality. 

Finally, suppose that $K_i$ is $p_i$-trivial, $p_i\ge 2$, $i=1,2$. 
In \eqref{eq1}, $S_1$ and $S_2$ are disks, and 
$$\chi (S) = p_1 + p_2 - p_1p_2 \le 0\ .$$
Hence 
\begin{equation}\label{eq8}
\begin{split}
\| K\| \le \eta (S)/p_1p_2 
& = \frac{p_1p_2 - (p_1 +p_2)}{2p_1p_2}\\
\noalign{\vskip6pt}
& = \frac12 - \frac{(p_1+p_2)}{2p_1p_2}
\end{split}
\end{equation}
Now consider \eqref{eq2}, and let $S_i$ have $d_i$ disk components, $i=1,2$.
Note that 
\begin{equation}\label{eq9}
d_i p_i\le p\ ,\qquad i=1,2\ ;
\end{equation}
in particular, since $p_i \ge2$, $d_i \le p/2$, $i=1,2$. 
It follows that 
$$\chi (S) \le d_1 + d_2 -p \le 0\ .$$
Therefore
\begin{equation}\label{eq10}
\begin{split}
\|K\| =\eta (S)/p 
& = -\chi (S) /2p\\
&\ge \frac12 - \frac{(d_1 + d_2)}{2p}\ .
\end{split}
\end{equation}

Comparing \eqref{eq10} with \eqref{eq8} gives 
$$\frac{d_1}p + \frac{d_2}p \ge \frac1{p_1} + \frac1{p_2}\ .$$
On the other hand, by \eqref{eq9} we have $d_i/p \le 1/p_i$, 
$i=1,2$, and hence \eqref{eq8} is an equality, as desired. 
\end{proof}

\subsection{Examples of knots with small rational genus}\label{example_subsection}

We use Lemma~\ref{surgery_lemma} to construct examples of knots $K$ 
in $3$-manifolds $M$ with arbitrarily small (but non-zero) rational genus.
These examples illustrate cases from our main classification theorems, to be
proved in \S~\ref{hyperbolic_knots_section} and \S~\ref{general_knots_section}, especially
case~(2) of Theorem~\ref{hyperbolic_complement_estimate}, 
case~(2) of Theorem~\ref{lens_theorem},
case~(3) of Corollary~\ref{hyperbolic_corollary}, case~(3) of Theorem~\ref{small_sfs_theorem}, and
cases~(3) and (4) of Theorem~\ref{toroidal_theorem}. We recall the notation $\Delta(\cdot,\cdot)$ for the minimal
number of points of intersection (i.e.\ the homological intersection number) 
of two unoriented isotopy classes of simple essential loops on a torus; 
in co-ordinates, $\Delta(a/b,c/d) = |ad-bc|$.

\smallskip

{\noindent \bf Case A.} ($M$ is hyperbolic and $K$ is the core of a Margulis tube)
Let $K'$ be a hyperbolic knot in $S^3$. 
Let $M$ be the result of $m/n$-Dehn surgery on $K'$, $m>0$, and let $K$ 
be the core of the surgery solid torus. 
By Lemma~\ref{surgery_lemma}, $\|K\| = \|K'\|/m\to0$ as $m\to\infty$. 
Also, for $m$ sufficiently large $M$ is hyperbolic and $K$ is a geodesic 
in $M$ whose length $\to 0$ as $m\to\infty$.
\smallskip

{\noindent \bf Case B.} ($M$ is a lens space and $K$ lies on a genus 1 Heegaard 
surface for $M$) Let $(M,K)$ be as in Case~A above, but with $K'$ a $(u,v)$-torus knot 
in $S^3$, where $u,v>1$. 
If $d= \Delta (m/n,\, uv/1) = |m-nuv| =1$ then $M$ is a lens space, 
and one sees that $K$ lies on a Heegaard torus in $M$. 
By choosing $m$ large enough we can make $\|K\| =\|K'\|/m$ arbitrarily small.
\smallskip

{\noindent \bf Case C.} ($M$ is a Seifert fiber space and $K$ is an ordinary or 
exceptional fiber) By taking $d>1$ in Case~B above, $M$ becomes a Seifert fiber space with base 
orbifold $S^2$ with three cone points of orders $u,v$ and $d$, and $K$ 
is the exceptional fiber of multiplicity $d$.

More generally, let $M'$ be a Seifert fibered homology 3-sphere, with 
base $S^2$ and $k\ge3$ exceptional fibers. 
Let $K'$ be an ordinary fiber, and let $a/b$ be the slope of the fiber 
on $\partial X'$, where $X' = M' - \inte N(K')$. 
Let $M$ be $m/n$-Dehn surgery on $K'$, and $K$ the core of the surgery 
solid torus.
Let $d = \Delta (a/b,\, m/n) = |an-bm|$. 
If $d=1$, then $M$ is a Seifert fiber space over $S^2$ with $k$ exceptional 
fibers, and $K$ is an ordinary fiber. 
If $d>1$, then $M$ is a Seifert fiber space over $S^2$ with $(k+1)$ exceptional 
fibers, and $K$ is an exceptional fiber of multiplicity $d$.
By Lemma~\ref{surgery_lemma}, $\|K\| = \|K'\|/m$, and in both cases this can be 
made arbitrarily small by taking $m$ sufficiently large. 
\smallskip

{\noindent \bf Case D.} ($M$ is hyperbolic and $K$ is a non-trivial cable of the 
core of a Margulis tube) Let $K_0$ be a hyperbolic knot in $S^3$. 
Fix coprime integers $p,q>1$, and let $k$ be any positive integer. 
Then $kq^2$ and $(1+kpq)$ are coprime, so there exist integers $a,b$ 
such that $akq^2 - b(1+kpq)=1$. 
Let $M$ be the manifold obtained by $(-kq^2/b)$-Dehn surgery on $K_0$, 
and let $K'\subset M$ be the core of the surgery. 
Let $X_0$ be the exterior of $K_0$; then $M= X_0\cup V$ where $V$ is a 
tubular neighborhood of $K'$. 
Let $\mu_0,\lambda_0$ (resp.\ $\mu,\lambda$) be a canonical 
meridian-longitude pair of generators for $H_1 (\partial X_0)$ 
(resp.\ $H_1 (\partial V)$). 
Let $f:\partial V \to \partial X_0$ be the gluing homeomorphism. 
We can choose $f$ so that with respect to the above bases 
$f_* : H_1 (\partial V) \to H_1 (\partial X_0)$ is given by the matrix
$$A = \begin{pmatrix} kq^2 & - (1+kpq)\\ \noalign{\vskip6pt}
-b&a\end{pmatrix}$$
Then 
$$A^{-1}  = \begin{pmatrix} a & 1+kpq\\ \noalign{\vskip6pt}
b&kq^2\end{pmatrix}$$
in particular, under the gluing the longitude $\lambda_0$ on 
$\partial X_0$ is identified with a curve of slope 
$(1+kpq)/kq^2$ on $\partial V$.

Let $K\subset \inte V$ be the $(p,q)$-cable of $K'$, the core of $V$.
Then $X = M -\inte N(K) = X_0\cup (V-\inte N(K)) = X_0\cup C$, 
where $C$ is a {\em $(p,q)$-cable space\/} (see \cite{Gordon_Litherland}, \S~3). 
There exists a planar surface $P\subset C$ with one boundary component on 
$\partial V$, with slope $(1+kpq)/kq^2$, and $q$ boundary components on 
$\partial N(K)$, with slope $(1+kpq)/k$ (see \cite[Lemma~3.1]{Gordon_Litherland}). 
Let $S_0$ be a Seifert surface for $K_0$, and define $S = P\cup S_0$, 
glued along $\partial S_0 = P\cap \partial V$. 
Since each of the $q$ components of $S\cap \partial N(K)$ has intersection 
number $k$ with the meridian of $K$, it follows that $S$ is a $qk$-Seifert surface for $K$.
Therefore $\|K\| \le \eta(S)/qk = (\eta (S_0) + (q-1)/2)/qk$, which goes to $0$
as $k$ goes to infinity. Also, for $k$ sufficiently large, $M$ is hyperbolic and $K'$ is the core of 
a Margulis tube in $M$.
\smallskip

{\noindent \bf Case E.} ($M$ is a Seifert fiber space and $K$ is a cable of an 
exceptional fiber but not a fiber) Repeat the construction in Case D above, 
but with $K_0$ a $(u,v)$-torus knot. 
Then $M$ is a Seifert fiber space and $K'$ is a fiber of multiplicity
$\Delta (-kq^2/b,\, uv/1) = |kq^2 + buv|$, which we can arrange to be $>1$.
Then $K$ is the $(p,q)$-cable of an exceptional fiber. 
On the other hand it is easy to see that $K$ can be a fiber in the Seifert fibration of $M$ for
at most one value of $k$.
\smallskip

{\noindent \bf Case F.} ($M$ is toroidal and $K$ lies in a torus in the JSJ decomposition of $M$) 
Let $K_1,K_2$ be non-trivial knots in $S^3$, and let $K'= K_1\,\#\, K_2$.
Let $M$ be the manifold obtained by $m$-surgery on $K'$ for some $m>0$, 
and let $K$ be the core of the surgery  solid torus. 
By Lemma~\ref{surgery_lemma}  $\|K\| = \|K'\|/m$; thus $\|K\|$ is non-zero but can be made 
arbitrarily small by taking $m$ sufficiently large. 

Let $X_i = S^3 -\text{int }N(K_i)$ be the exterior of $K_i$, $i=1,2$; 
then $X = S^3 - \text{int }N(K') = X_1\cup_A X_2$ where $A$ is a meridional 
annulus in $\partial X_i$, $i=1,2$. 
Let $V$ be the surgery solid torus.
Note that the boundary slope of $A$ on $\partial X$ (the meridian of $K'$) 
intersects the meridian of $V$ once. 
It follows that $X_1\cup V \cong X_1$, and so $M = (X_1\cup X_2) \cup V 
\cong X_1 \cup_T X_2$, where $T = \partial X_1 = \partial X_2$. 
Also $K$, the core of $V$, is isotopic into $T$.

If, for example, we take $K_1$ to be hyperbolic and $K_2$ to be either 
hyperbolic or a torus knot, then $T$ is the unique torus in the JSJ 
decomposition of $M$. Note also that in the second case $K$ is not a Seifert fiber in $X_2$.
If we take $K_i$ to be the $(p_i,q_i)$ torus knot, $i=1,2$, and 
$m\ne p_1q_1 + p_2q_2$, then again $T$ is the unique JSJ torus in $M$, and $K$ is not a Seifert fiber of either $X_1$ or $X_2$.

\medskip

Further examples are given in the next subsection.

\subsection{Torus bundles}\label{torus_bundles_subsection}

In this section we analyze the case where $M$ is a $T^2$-bundle over $S^1$ 
and $K$ is an essential simple closed curve in a fiber. 
It turns out that this gives further examples of knots with arbitrarily 
small rational genus, which need to be taken account of in the statement 
of Theorem~\ref{toroidal_theorem}.

Let $f:T^2 \to T^2$ be an orientation preserving diffeomorphism, and let 
$M_f$ be the mapping torus of $f$, obtained from $T^2\times I$ by 
identifying $(x,0)$ with $(f(x),1)$ for all $x\in T^2$. 
If we choose a basis for  $H_1(T^2)$, the automorphism $f_*$ of $H_1(T^2)$ 
induced by $f$ is represented by a matrix $A_f\in \SL(2,\Z)$; 
the diffeomorphism type of $M_f$ depends only on the conjugacy class 
of $A_f$ in $\GL(2,\Z)$. 
Let $K$ in $M_f$ be an essential simple closed curve in a fiber. 
If $\text{trace }f_*\ne2$ then $\det (A_f-I) \ne0$, and so every such $K$ 
has finite order in $H_1(M_f)$.
If $\text{trace }f_* =2$ then $A_f$ is conjugate in $\GL(2,\Z)$ to 
$\left(\begin{smallmatrix} 1&p\\ 0&1\end{smallmatrix}\right)$ 
for some $p\ge 0$.
If $p=0$ then $M_f = T^3$ and no $K$ has finite order in $H_1(M_f)$, 
so in the sequel we shall always assume that $f_* \ne \Id$. 
If $p\ge 1$ then there is a unique $K$ with finite order $(=p)$ in $H_1(M_f)$.

Let $a$ and $b$ be oriented simple closed curves in $T^2$ meeting 
transversely in a single point, such that $K$ is the image in $M_f$ of the 
curve $a\times\{1/2\}\subset T^2\times I$.
Thinking of $T^2\times I$ as $(a\times b) \times I = a\times (b\times I)$ 
shows that $(T^2\times I,K) = (S^1\times A^2,\, S^1\times\{\text{point}\})$,
and therefore $T^2\times I - \inte N(K) \cong S^1\times P^2$,
where $P^2$ is a pair of pants. 
Let the boundary components of $P^2$ be $B_0$, $B_1$ and $C$, where 
$S^1\times B_i = T_i = T^2\times \{i\}$, $i=0,1$, and 
$S^1\times C = \partial N(K)$. 
Let $b_0,b_1,c$ be the homology classes of $B_0$, $B_1$ and $C$, respectively,
oriented so that $[\partial P^2] = b_0 - b_1 +c$, and so that $b_0$ and $b_1$ 
map to the class $b$ above in $H_1 (T^2\times I)$ 
(we will abuse notation by not distinguishing between $a,b$ and their 
classes in $H_1 (T^2)$).

We now wish to describe certain horizontal surfaces in $S^1\times P^2$.
Consider the homomorphism 
$H_1(P^2)\to\Z$ defined by $b_0\mapsto \ell_0$, $b_1\mapsto \ell_1$ 
(so $c\mapsto \ell_1-\ell_0$), where $\ell_0$ and $\ell_1$ are arbitrary integers. 
Composing with the Hurewicz map we get a homomorphism $\pi_1(P^2)\to\Z$,
which is induced by a map $\pi :P^2 \to S^1$. 
Let $m$ be an integer $\ge1$ and let $\sigma :S^1\to S^1$ be the 
connected covering of degree~$m$. 
Let $F\to P^2$ be the $\Z/m\Z$ covering corresponding to the 
composition of $\pi_*$ with the quotient map $\Z \to \Z/m\Z$. 
This covering is the pull-back of $\sigma$ under the map $\pi$, in 
other words we have a commutative diagram
$$ \begin{CD}
F @>p_1>> S^1 \\
@Vp_2VV @VV\sigma V \\
P^2 @>>\pi> S^1
\end{CD}$$
where $F$ is identified with the subset 
$\{(x,y) : \sigma (x) =\pi (y)\}$ of $S^1\times P^2$, and $p_1$ and $p_2$ 
are the restrictions of the projections of $S^1\times P^2$ onto the factors. 

Consider the ordered bases $(a,b_0)$, $(a,b_1)$ and $(a,c)$ for 
$H_1 (T_0)$, $H_1(T_1)$ and $H_1 (\partial N(K))$, respectively. 
Let $\tilde B_0$, $\tilde B_1$, $\tilde C$ be the respective inverse 
images in $\partial F$ of $B_0$, $B_1$ and $C$. 
Since the coverings $\tilde B_i \to B_i$, $i=0,1$, and $\tilde C\to C$ 
have degree~$m$, we have
\begin{equation}\label{eq1.0}
\begin{array}{rcl}
[\tilde B_i] & = & \ell_i a + m b_i\ ,\qquad i=0,1\ ,\cr
\noalign{\vskip6pt}
[\tilde C] & = & \ell a + mc\ ,\ \text{ where }\ \ell = \ell_1 - \ell_0\ .
\end{array} 
\end{equation}
In particular, if $\ell\ne 0$ then $F$ is a relative $|\ell|$-Seifert surface
for $K$ in $T^2\times I$. 

Let $f_*$ be represented with respect to the basis $(a,b)$ by 
$\left(\begin{smallmatrix}\alpha &\beta \\ \gamma&\delta\end{smallmatrix}
\right) \in \SL(2,\Z)$. 
Assume for the moment that we are not in the case $\gamma=0$, 
$\alpha=\delta = -1$. 
Then we can choose  $\ell_0$ and $m\ge1$ such that 
\begin{equation}\label{eq1.1}
\gamma \ell_0 + (\delta -1) m=0\ ,
\end{equation}
and define 
\begin{equation}\label{eq1.2}
\ell_1 = \alpha \ell_0 + \beta m\ ; 
\end{equation}
so 
\begin{equation}\label{eq1.3}
\ell = (\alpha -1) \ell_0 +\beta m\ .
\end{equation}
Then 
$\left(\begin{smallmatrix}\alpha&\beta\\ \gamma&\delta\end{smallmatrix}\right)
\left(\begin{smallmatrix}\ell_0\\ m\end{smallmatrix}\right) = 
\left(\begin{smallmatrix}\ell_1\\ m\end{smallmatrix}\right)$, 
which implies that we may isotope $f$ so that $f(\tilde B_0) =\tilde B_1$, 
and hence $F$ becomes an orientable surface $S$ in $M_f$. 
If $\text{trace }f_* \ne2$ then \eqref{eq1.1} and \eqref{eq1.3}
imply that $\ell \ne0$, and so $S$ is an $|\ell|$-Seifert surface for $K$.
Since $\chi (S) = \chi(F) = m\chi (P^2) = -m$, we get 
\begin{equation}\label{eq1.4}
\| K\| \le m/2|\ell|\ .
\end{equation}
We note that if $\text{trace }f_* \ne2$ and $\gamma\ne0$, it follows 
easily from \eqref{eq1.1} and \eqref{eq1.3} that
\begin{equation}\label{eq1.5}
m/\ell = \gamma / (\text{trace }f_*-2)\ .
\end{equation}

In the case $\text{trace }f_* =2$, $A_f$ is conjugate to 
$\left(\begin{smallmatrix} 1&p\\ 0&1\end{smallmatrix}\right)$, $p\ge 1$, 
where the first element of the corresponding ordered basis for $H_1 (T^2)$ 
is represented by the unique $K$ that has finite order in $H_1(M_f)$. 
Thus $\left(\begin{smallmatrix}\alpha &\beta\\ \gamma&\delta\end{smallmatrix}
\right) = \left(\begin{smallmatrix} 1&p\\ 0&1\end{smallmatrix}\right)$, 
and in \eqref{eq1.1} we choose $\ell_0 =0$, $m=1$, giving $\ell=\ell_1 =p$.

The above discussion shows that unless $\gamma=0$ and $\text{trace }f_*=-2$,
any $K$ in a fiber of $M_f$ having finite order in $H_1(M_f)$ has 
an $|\ell|$-Seifert surface, $\ell\ne0$, such that the corresponding 
surface $F$ in $T^2 \times I- \inte N(K) = S^1\times P^2$ 
is horizontal (i.e.\/ transverse to the $S^1$ fibers), and \eqref{eq1.4} holds.

It remains to discuss the case $\gamma =0$, $\text{trace }f_* = -2$, i.e. 
where our matrix is 
$\left(\begin{smallmatrix} -1&\beta\\ 0&-1\end{smallmatrix}\right)$. 
Let $A_i$ be a vertical annulus in $S^1\times P^2$ with one boundary 
component on each of $T_i$ and $\partial N(K)$, $i=0,1$. 
Since $f_* (a) = -a$, $A_0$ and $A_1$ glue up to give an annulus $A$ 
in $M_f - \inte N(K)$ whose boundary components are coherently 
oriented on $\partial N(K)$. 
Hence $\|K\| =0$.

Finally, we show that the inequality \eqref{eq1.4} is an equality.
Let $S$ be a good $p$-Seifert surface for $K$ in $M_f$ such that 
$\|K\| = \eta (S)/p$.
Let $T\subset M_f$ be the fiber that is the image of $T^2\times\partial I$; 
note that $K\cap T = \emptyset$. 
Isotoping $S$ to minimize the number of components of $S\cap T$ we get 
a relative $p$-Seifert surface $F$ for $K$ in $T^2\times I$ that is 
essential in $T^2\times I-\inte N(K) = S^1\times P^2$. 
Therefore $F$ is either horizontal or vertical. 

First we dispose of the vertical case. 
Here $F$ must consist of either a single annulus with both boundary 
components on $\partial N(K)$, or two annuli, one running from 
$\partial N(K)$ to $T_0$ and the other from $\partial N(K)$ to $T_1$. 
In the first case, $[\partial S] = [\partial F] =0\in H_1(\partial N(K))$, 
a contradiction. 
In the second case, we also get $[\partial S] =0\in H_1(\partial N(K))$ 
unless $f_*(a)=-a$, in which case $A_0$ and $A_1$ glue up to give an 
annulus $A$ in $M_f -\inte N(K)$ that is a 2-Seifert surface for $K$, 
implying that $\|K\|=0$.
This is precisely the case $\gamma=0$, $\text{trace }f_*=-2$ 
discussed above. 

Now suppose $F$ is horizontal. 
Then the restriction to $F$ of $p_2 :S^1\times P^2 \to P^2$ is a covering 
projection, of degree $m\ge1$, say. 
Then, with the same notation as used earlier, we see that 
\eqref{eq1.0} must hold, and the subsequent discussion shows that \eqref{eq1.1} 
holds for some $\ell_0$, that $p=|\ell|$ where $\ell = \ell_1 - \ell_0$, 
and thence that $\|K\| = \eta (S) /|\ell | = m/2|\ell |$. 

The following theorem summarizes our conclusions.
\begin{theorem}\label{torus_bundle_theorem}
Let $M_f$ be a $T^2$-bundle over $S^1$ with monodromy $f$ not 
isotopic to the identity, and let $K$ be an essential simple closed curve 
in a fiber.
\begin{enumerate}
\item{If $\text{trace }f_* =2$ then there is a unique $K$ that has finite order 
in $H_1(M_f)$, and $\|K\| = 1/2p$ where $f_*$ is represented by the matrix 
$\left(\begin{smallmatrix}1&p\\ 0&1\end{smallmatrix}\right)$, $p\ge1$.}
\item{If $\text{trace }f_*\ne 2$ then every $K$ has finite order in $H_1(M_f)$, 
and $\|K\| = |\gamma/2(\text{trace }f_*-2)|$ where $f_*$ is represented 
by the matrix 
$\left(\begin{smallmatrix}\alpha&\beta\\ \gamma&\delta\end{smallmatrix}\right)$
with respect to an ordered basis of $H_1(T^2)$ whose first member is $[K]$.}
\end{enumerate}
\end{theorem}

We see immediately from Theorem~\ref{torus_bundle_theorem}
that knots in fibers of torus bundles provide additional examples of knots 
with arbitrarily small rational genus. These examples are relevant to Theorem~\ref{toroidal_theorem},
cases (\ref{Seifert_case}) and (\ref{torus_case}). 
We now describe this in more detail, continuing the list in \S~\ref{example_subsection}.

\smallskip

{\noindent \bf Case G.} ($M$ is a Seifert fiber space with no exceptional fibers and $K$ is a fiber)
Let $M$ be $M_f$ where $\text{trace }f_* =2$, as in part~(1) of 
Theorem~\ref{torus_bundle_theorem}.
Then $M$ can also be described as an $S^1$-bundle over $T^2$ with Euler 
number $p$, and $K$ is a fiber. We remark that $M$ has a Nil geometric structure
(see \cite{Scott} pp. 467--470). 
By taking $p$ large we can make $\|K\| = 1/2p$ arbitrarily small. 
\smallskip

{\noindent \bf Case H.} ($M$ is a $T^2$-bundle over $S^1$ with Anosov monodromy and $K$ lies in a 
fiber) 
Let $M$ be $M_f$ with $|\text{trace }f_*| >2$, so the monodromy $f$ is Anosov.
Note that $M$ has a Sol geometric structure (see \cite{Scott} pp. 470--472).
It is clear  from the formula in part~(2) of Theorem~\ref{torus_bundle_theorem}
that we can choose $M$ and $K$ so that $\|K\|$ is arbitrarily small. 
For example, choose any matrix 
$\left(\begin{smallmatrix}\alpha&\beta\\ \gamma&\delta\end{smallmatrix}
\right) \in \SL(2,\Z)$ with $\gamma \ne 0$ and let $f_n :T^2\to T^2$ 
be given by the matrix 
$\left(\begin{smallmatrix} 1&n\\ 0&1\end{smallmatrix}\right) 
\left(\begin{smallmatrix} \alpha&\beta\\ \gamma&\delta\end{smallmatrix}\right)
= \left(\begin{smallmatrix} \alpha+n\gamma &\beta+n\delta\\ 
\gamma &\delta \end{smallmatrix}\right)$ 
with respect to some basis $(a,b)$, say. 
Let $M_n$ be the corresponding $T^2$-bundle and let $K_n$ be the knot 
in a fiber such that $[K_n]=a$. 
Then $\|K_n\| = |\gamma/2(\alpha +\delta +n\gamma -2)| \to 0$ as 
$n\to\infty$.

\begin{remark}\label{torus_bundle_remark}
It follows from Theorem~\ref{torus_bundle_theorem}
that if we are not in Case~G or Case~H, i.e. if $\text{trace }f_*=-2$ 
or $|\text{trace }f_*| \le 1$, then either $\|K\| =0$ or $\|K\|\ge 1/8$.
\end{remark}

\subsection{Knots in vertical tori in Seifert fiber spaces}	
In this section we analyze relative $p$-Seifert surfaces for knots that 
lie in essential vertical tori in Seifert fiber spaces. 
First we have the following lemma.
\bigskip

\begin{lemma}\label{lem2.23}
Let $M$ be a Seifert fiber space with non-empty boundary and let 
$\pi :M\to \B$ be the projection of $M$ onto its base orbifold $\B$. 
Let $F$ be a horizontal surface in $M$ and let $k$ be the degree of the 
induced branched covering $\pi\mid F:F\to \B$. 
If $\chi (F) <0$ then $\chi (F) \le -k/6$.
\end{lemma}

\begin{proof} 
Let $q_1,\ldots,q_n$ be the multiplicities of the exceptional fibers of $M$.
Then (see for example \cite{Hatcher_notes}, \S~2.1)
\begin{equation*}
\chi (F) = k\bigg( \chi (\B) - \sum_{i=1}^n \Big(1-\frac1{q_i}\Big)\bigg)\ .
\end{equation*}
Since $\chi (F) <0$, the maximal value of the expression in parentheses is
attained when $\chi(\B)=1$, $n=2$, $q_1=2$, $q_2=3$, which gives the 
value $-1/6$.
\end{proof}

\begin{proposition}\label{prop2.24}  
Let $M$ be a Seifert fiber space, $T$ a vertical essential torus in $M$, 
$K$ an essential simple closed curve in $T$, and $F$ a relative  $p$-Seifert 
surface for $K$. 
Then either 
\begin{enumerate}
\item{$-\chi^- (F) \ge p/6$; or}
\item{$K$ is an ordinary fiber in the Seifert fibration of $M$; or}
\item{$M$ contains a submanifold $Q$ that is a twisted $S^1$-bundle 
over the M\"obius band and $K$ is a fiber in $Q$.}
\end{enumerate}
\end{proposition}

Note that if (3) holds but (2) does not then the Seifert fibration      
of $Q$ induced from $M$ is the one with base orbifold a disk with two 
cone points of order~2.

\begin{proof}
Let $T^2\times I$ be a regular neighborhood of $T= T^2 \times \{1/2\}$. 
First suppose that $T$ separates $M$.
Then $M = X_0 \cup T^2 \times I \cup X_1$, where $T_i = T^2\times\{i\}$ 
is a component of $\partial X_i$, $i=0,1$. 
Let $N(K)$ be a regular neighborhood of $K$ in $T^2 \times I$ and let 
$Y = T^2 \times I - \inte N(K)$. 
Then $X = M-\inte N(K) = X_0 \cup Y\cup X_1$. 
Let $F$ be a good relative $p$-Seifert surface for $K$ in $M$. 
Let $F_i = F\cap X_i$, $i=0,1$, and $G = F\cap Y$. 
We may  assume that $F_0$, $F_1$ and $G$ are essential in $X_0$, $X_1$ and $Y$ 
respectively. 
The Seifert fibration of $M$ induces Seifert fiber structures on $X_0$ and 
$X_1$, with base orbifolds $\B_0$ and $\B_1$, say.
Recall from \S~\ref{torus_bundles_subsection} that $Y \cong S^1 \times P^2$ where $P^2$ is a 
pair of pants. 

Note that $F_i$ is horizontal or vertical in $X_i$, $i=0,1$, and $G$ is 
horizontal or vertical in $Y\cong S^1\times P^2$. 
Write 
$\partial_i F_i = \partial F_i \cap T_i = \partial G\cap T_i = \partial_i G$,
$i=0,1$.

\smallskip

{\noindent \bf Case I. $G$ is vertical}

Since a vertical annulus in $Y$ that has both its boundary components 
on $\partial N(K)$ has these boundary components oriented oppositely on 
$\partial N(K)$, and since $F$ is good, it follows that $G$ consists of 
$p$ parallel copies of an annulus with one boundary component on 
$\partial N(K)$ and the other on (say) $T_0$, $F_1=\emptyset$, and $F_0$ 
is connected.

\smallskip

{\noindent \bf Subcase (a).} $\chi (F_0)<0$

Then $F_0$ is horizontal in $X_0$. 
Since $\partial_0 F_0$ has $p$ components the index of the covering 
$F_0\to \B_0$ is at least $p$. 
Therefore, by Lemma~\ref{lem2.23}, 
$\chi (F) = \chi (F_0) \le -p/6$.

\smallskip

{\noindent \bf Subcase (b).} $\chi (F_0) =0$

Then $F_0$ is an annulus. 
First suppose that $\partial_0 F_0$ has a single component.
If $F_0$ is horizontal then $X_0 \cong F_0 \times S^1\cong T_0\times I$, 
contradicting the assumption that $T$ is essential in $M$.
If $F_0$ is vertical then 
$K$ is an ordinary fiber in the Seifert fibration of $M$.

If $\partial_0 F_0$ has two components then $F$ is an annulus with both 
boundary components on $\partial N(K)$ and by Theorem~\ref{vanishing_theorem}
$K$ is contained in a submanifold $N$ of $M$ where $N$ is a Seifert 
fiber space over the M\"obius band with at most one orbifold point 
of order $r\ge 1$ and $K$ is a fiber of multiplicity $r$.
If $r > 1$ then, since the Seifert fibration of $N$ is unique, $K$ is 
an exceptional fiber in $M$. 
But this contradicts the fact that $K$ is contained in a vertical torus. 
Hence $r = 1$ and we have conclusion (3). 

This completes the proof in Case I.

\smallskip

{\noindent \bf Case II. $G$ is horizontal}

Here we will adopt the notation of \S~\ref{torus_bundles_subsection}. 
Let $m$ be the index of the covering $G\to P^2$; so $\chi (G) = - m$. 
Also we have 
\begin{gather*}
[\partial_i F_i] = [\partial_i G] = \ell_i a + mb_i\ ,\qquad i=0,1\ ,\\
[\partial G\cap N(K)] = (\ell_1 - \ell_0) a+mc\ ,
\end{gather*}
where $p = |\ell_1 - \ell_0|$.

Let $\varphi_i$ be the Seifert fiber of $M$ on $T_i$, $i=0,1$. 
Then $[\varphi_i] = \alpha a +\beta b_i$, say, $i=0,1$. 
If $\beta =0$ then $[K] = [\varphi_i]$ so $K$ is an ordinary fiber in 
the Seifert fibration of $M$.
We will therefore assume that $\beta \ne 0$.

If $F_i$ is horizontal in $X_i$, let $k_i$ denote the index of the 
associated covering $F_i \to \B_i$, $i=0,1$. 
Then $k_i = |\partial_i F_i \cdot \varphi_i| = |(\ell_i a + mb_i)\cdot 
(\alpha a + \beta b_i)| = |\beta \ell_i - \alpha m|$, where $\cdot$
denotes algebraic intersection number.

\begin{sublemma}\label{sublem2.25} 
\begin{enumerate}
\item{If both $F_0$ and $F_1$ are horizontal then $k_0 + k_1 \ge p$.}
\item{If $F_i$ is horizontal and $\chi (F_i)=0$ then $k_i\le m$.}
\item{If $F_0$ is horizontal and $F_1$ is vertical then $k_0 \ge p$.}
\end{enumerate}
\end{sublemma}
\begin{proof} 

(1) $k_0 +k_1  = |\beta \ell_0 - \alpha m| + |\beta \ell_1 -\alpha m|
\ge |\beta|\, |\ell_1 - \ell_0| = |\beta| p \ge p$. 

(2) 
Here $F_i$ consists of parallel copies of a horizontal annulus $A$ in $X_i$.
If $A$ has one boundary component on $T_i$ and one on $\partial M$ then 
$X_i \cong A\times S^1 \cong T_i\times  I$, contradicting the assumption 
that $T$ is essential in $M$. 
Also, since the components of $\partial_i G$ are coherently oriented on 
$T_i$ the same holds for $\partial_i F_i$. 
It follows that $A$ is non-separating and $\B_i$ is a disk with two 
cone points of order~2.
In particular each boundary component of $A$ has intersection number~1 with 
the Seifert fiber $\varphi_i$. 
Therefore $k_i= |\partial_i F_i \cdot \varphi_i|$ is the number of 
components of $\partial_i F_i = \partial_i G$, which is $\le m$ since $G$ 
is an $m$-fold covering of $P^2$.

(3) 
If $F_1$ is vertical then $[\partial_1 F_1] = \ell_1 a + mb_1 = s[\varphi_1] 
= s(\alpha a +\beta b_1)$ where $s$ is the number of components of 
$\partial_1 F_1 = \partial_1 G$. 
Hence $s\le m$. 
Now 
\begin{equation*} 
\begin{split} 
sk_0 & = |\partial_0 F_0 \cdot s\varphi_0|\\
& = (\ell_0 a + mb_0) \cdot (\ell_1 a + mb_0)\\
& = m|\ell_1 - \ell_0| = mp\ .
\end{split}
\end{equation*}
Therefore $k_0 = mp/s\ge p$.
\end{proof}

We now complete the proof of Proposition~\ref{prop2.24} in Case~II.

First note that $F_0$ and $F_1$ cannot both be vertical, for then we 
would have $\ell_0 =\ell_1$ and hence $p=0$.

\smallskip

{\noindent \bf Subcase (a).} $F_0$ and $F_1$ horizontal

If $\chi (F_0),\chi(F_1) < 0$ then 
\begin{equation*}
\begin{split} 
|\chi (F)| & = |\chi (F_0)| + |\chi (F_1)| + |\chi (G)|\\
& \ge (k_0 + k_1)/6 +m\ ,\ \text{ by Lemma~\ref{lem2.23}}\\
& \ge p/6 + m,\ \text{ by Sublemma~\ref{sublem2.25} (1)}\\
& > p/6\ .
\end{split}
\end{equation*}

If $\chi (F_0) =0$ and $\chi (F_1) < 0$ then $k_0 \le m$ by 
Sublemma~\ref{sublem2.25} (2), and hence $k_1 \ge p-m$ by 
Sublemma~\ref{sublem2.25} (1). 
Therefore 
\begin{equation*}
\begin{split} 
|\chi (F)| & = |\chi (F_1)| + |\chi (G)|\\
& \ge k_1/6 + m\\
& \ge (p-m)/6 + m > p/6\ .
\end{split}
\end{equation*}
Finally, if $\chi (F_0) = \chi (F_1) =0$, then by Sublemma~\ref{sublem2.25}
parts (1) and (2), we have $p \le k_0 + k_1 \le 2m$.
Hence 
\begin{equation*}
|\chi (F)| = |\chi (G)| = m \ge p/2\ .
\end{equation*}

\smallskip

{\noindent \bf Subcase (b).} $F_0$ horizontal, $F_1$ vertical

If $\chi (F_0) < 0$ then 
\begin{equation*}
\begin{split} 
|\chi (F)| & = |\chi (F_0)| + |\chi (G)|\\
& \ge k_0/6 + m\\
& \ge p/6 + m\ ,\ \text{ by Sublemma~\ref{sublem2.25} (3)}\\
& > p/6\ .
\end{split}
\end{equation*}

If $\chi (F_0) =0$ then by Sublemma~\ref{sublem2.25} (2) $k_0 \le m$, 
while by Sublemma~\ref{sublem2.25} (3) $k_0 \ge p$. 
Hence $|\chi (F)| = |\chi (G)| = m \ge p$.

This completes the proof of Proposition~\ref{prop2.24} when $T$ separates $M$. 

If $T$ is non-separating, let $\pi :M\to \B$ be the projection from $M$ to its 
base orbifold $\B$. Let $N$ be a regular neighborhood of either the union of 
the exceptional fibers of $M$ or, if $M$ is closed and has no exceptional 
fibers, an ordinary fiber. Let $M_0 = M - \inte N$, with corresponding base 
orbifold $\B_0$. Then $M_0$ is an $S^1$-bundle over $\B_0$ and $T = 
\pi^{-1}(C)$ for some non-separating orientation-preserving simple closed 
curve $C$ in $\B_0$. Now $H_1(T)$ has basis $\varphi, \gamma$, where $\varphi$ is the
class of the $S^1$-fiber of $M_0$ and $\pi_{\ast}(\gamma) = [C] \in H_1(\B_0)$.
Therefore $[K] = r\varphi + s\gamma$ for some pair of relatively prime integers
$r,s$. 

Isotoping $F$ to be transverse to the core of the components of $N$, let 
$F_0 = F\cap M_0$. Then $F_0$ defines a homology of $pK$ into 
$\partial M_0$. Therefore, considering the map $\pi_{\ast}:H_1(M_0, \partial 
M_0) \to H_1(\B_0, \partial\B_0)$, we have $0 = \pi_{\ast}(p[K]) = ps[C]$.
Since $C$ is orientation-preserving and non-separating, $[C]$ has infinite 
order in $H_1(\B_0, \partial\B_0)$, and so we conclude that $s = 0$. Therefore 
$K$ is an ordinary fiber in the Seifert fibration of $M$. 
\end{proof}

\section{Graphs}\label{graph_section}

Several of our arguments concern the interaction between a relative $p$-Seifert surface $F$ 
for a knot $K$ in $M$, and another surface $\widehat{G}$ properly embedded in $M$. Such arguments
are handled in a uniform manner, by making the surfaces intersect as simply as possible, and
then by analyzing cases depending on the combinatorics of this intersection.
The arguments in this section are mostly combinatorial.

\subsection{Graphs on surfaces}

Fix the following notation. Let $K$ denote a knot in $M$, let $F$ be a good relative $p$-Seifert surface for $K$,
and let $\widehat{G}$ be a properly embedded surface in $M$ (usually a Heegaard surface, or an
essential surface; usually of low complexity). Under such a circumstance, we perform the following procedure.

Isotop $N(K)$ so that it meets $\widehat{G}$ in $n$ meridian disks, and let $G = \widehat{G} \cap X$, so that $F$ and $G$
are both proper surfaces in $X$. After an isotopy, we may assume that $F$ and $G$ meet transversely in a finite
disjoint union of circles and properly embedded arcs and that each of the $q$ components of $\partial F \cap \partial N(K)$
meets each of the $n$ components of $\partial G \cap \partial N(K)$ in $r$ points, with notation as in 
Definition~\ref{good_definition}.

Formally cap off the components of $\partial F \cap \partial N(K)$ with disks to obtain
a surface $\widehat{F}$ (note: if $F$ is an honest $p$-Seifert surface, then $\widehat{F}$ is closed. Otherwise,
$\partial \widehat{F} = \partial F \cap \partial M$). The intersection $F \cap G$ determines graphs $\Gamma_F$ and
$\Gamma_G$ in $\widehat{F}$ and $\widehat{G}$ respectively, where the vertices of $\Gamma_F$ (resp. the vertices of 
$\Gamma_G$) correspond to the disks of $\widehat{F} - F$ (resp. the disks of $\widehat{G} - G$) and the edges correspond
to the arc components of $F \cap G$ with at least one endpoint on $\partial N(K)$. We distinguish between two kinds of
edges of $\Gamma_F$ and $\Gamma_G$: {\em interior edges}, which have both endpoints on $\partial N(K)$ (i.e.\ at the
vertices), and {\em boundary edges}, which have one endpoint on $\partial N(K)$, and the other on 
$\partial \widehat{F} = F \cap \partial M$ or $\partial \widehat{G} = G \cap \partial M$. 

Choose orientations on $F$, $G$ and $X$. This induces orientations on $\partial F$, $\partial G$ and $\partial X$, and
an arc of $F \cap G$ joins points of intersection of $\partial F$ with $\partial G$ of opposite sign.

Number the components of $\partial G \cap \partial N(K)$ (equivalently, the vertices of $\Gamma_G$) with the integers
$1,2,\cdots,n$ in the (cyclic) order they occur along $\partial N(K)$. Hereafter an {\em index} means an element
of this index set; i.e.\ an element $i \in \lbrace 1,\cdots,n\rbrace$.
We imagine that the vertices of $\Gamma_F$ and
$\Gamma_G$ are thickened, so that distinct edges end at distinct ``edge-endpoints'' on a (thickened) vertex. With this
convention, an edge-endpoint at a vertex of $\Gamma_F$ is a point of intersection of the corresponding component of
$\partial F \cap \partial N(K)$ with a component of $\partial G \cap \partial N(K)$, and we label the edge-endpoint
with the index corresponding to the label on the component of $\partial G \cap \partial N(K)$.
Notice that it is the {\em vertices} of $\Gamma_G$ and the {\em edge-endpoints} of
$\Gamma_F$ that are labeled with indices.
Since $F$ is good (by hypothesis), all components of $\partial F \cap \partial N(K)$ are coherently oriented on
$\partial N(K)$, and therefore at each vertex of $\Gamma_F$ we see the index labels $1,2,\cdots,n,1,2,\cdots,n,\cdots$ repeated
$r$ times in (say) anticlockwise order on the edge-endpoints around the vertex. Notice that if $\Gamma_F$ and
$\Gamma_G$ have $e_i$ interior edges and $e_\partial$ boundary edges, then $2e_i + e_\partial = pn$.

\medskip

In applications, the surface $\widehat{G}$ will always be either essential, or a Heegaard surface. In the former case we
will choose $n = |K \cap \widehat{G}|$ to be minimal; and in the latter case we will put $K$ in thin position.

\begin{remark}
Thin position for knots in $S^3$ was introduced by Gabai \cite{Gabai_3}, and for knots in arbitrary $3$-manifolds
by Rieck \cite{Rieck, Rieck_paper}, and we refer to these references for details. Technically, a knot is in thin
position with respect to a sweepout of a $3$-manifold (associated to a Heegaard splitting). The Heegaard surface $\widehat{G}$
is one of the nonsingular level sets of this sweepout, chosen depending on $F$.
\end{remark}

\begin{lemma}\label{no_monogons}
With notation and conventions as above, we can arrange that no arc component of $F \cap G$ with both endpoints in 
$\partial N(K)$ is boundary parallel in either $F$ or $G$. Equivalently, the graphs $\Gamma_F$ and $\Gamma_G$ 
have no monogon (disk) faces.
\end{lemma}
\begin{proof}
The arguments are standard.
Since components of $\partial F$ are oriented coherently on $\partial N(K)$, every point of intersection of a
given component of $\partial G$ with $\partial F$ has the same sign. Hence in particular, every interior
edge of $\Gamma_G$ has endpoints on distinct vertices of $\Gamma_G$, and there are never any complementary monogons.

If there is a monogon complementary to $\Gamma_F$ then $\widehat{G}$ 
can be pushed over such a monogon by an isotopy, thereby reducing the number of 
intersections with $K$; this is ruled out by hypothesis when $\widehat{G}$ is essential.

It remains to rule out monogon regions for $\Gamma_F$ when $\widehat{G}$ is a Heegaard 
surface (note that such monogons {\em may} contain interior loops of $F \cap \widehat{G}$ 
that bound compressing disks for $\widehat{G}$; see footnote~12 on page~635 of \cite{Rieck_paper}). 
Such a monogon region is either a {\em high} disk or a {\em low} disk for $\widehat{G}$, in the
terminology of \cite{Rieck_paper}. The existence of disjoint high and low disks at some level violates thinness;
Gabai's argument in \cite{Gabai_3} (also see Theorem~6.2 in \cite{Rieck_paper}) 
shows that for a knot in thin position, some level set of the sweepout admits
neither. Choosing $\widehat{G}$ to be such a level set, $\Gamma_F$ has no monogons.
\end{proof}

\begin{remark}
If $F\cap G$ has a simple closed curve component that bounds a disk in $G$, 
let $\gamma$ be an innermost such, i.e. $\gamma$ bounds a disk $D$ in $G$
such that $F\cap (\inte D)=\emptyset$. 
Since $F$ is incompressible in $X$, the loop $\gamma$ bounds a disk $E$ in $F$.
Surgering $F$ along $D$ produces a $2$-sphere $\Sigma = E\cup D$ together 
with a surface $F'$ that is essential in $X$, has $\partial F'=\partial F$, 
and is homeomorphic to $F$, and which may be isotoped so that 
$|F'\cap G| < |F\cap G|$. 
If $M$  is irreducible then $\Sigma$ bounds a $3$-ball and $F'$ is isotopic 
to $F$. 
So in that case we may assume that no simple closed curve component of 
$F\cap G$ bounds a disk in $G$. 

If $\widehat G$ is essential we will always assume that $n= |K\cap \widehat G|$ 
is minimal over all essential surfaces $\widehat G$ in $M$ of the given 
homeomorphism type.
This implies that $G$ is essential in $X$. 
Hence by the remarks above, interchanging the roles of $F$ and $G$, we may 
assume that no simple closed curve component of $F\cap G$ bounds a disk in $F$.
\end{remark}

Let $\Gamma_F$ as above be a graph on $\widehat F$ without monogons. 
If every complementary region to $\Gamma_F$ is a bigon, then either 
$\widehat F$ is a sphere, $\Gamma_F$ has exactly two vertices, with 
parallel interior edges running between them, or $\widehat F$ is a 
disk, $\Gamma_F$ has exactly one vertex, with parallel boundary edges 
running from the vertex to the boundary. 
We call such a $\Gamma_F$ a {\em beachball\/} (Figure~\ref{beachball_fig} indicates why), 
of the {\em first kind\/} and {\em second kind\/} respectively. 

\begin{remark}\label{beachball_remark}
If $\Gamma_F$ is a beachball of the first kind then $\|K\|=0$ and $K$ 
satisfies conclusion~(3) of Theorem~\ref{vanishing_theorem}.
\end{remark}

\begin{figure}[thbp]
\labellist
\small\hair 2pt
\endlabellist
\centering
\includegraphics[scale=0.5]{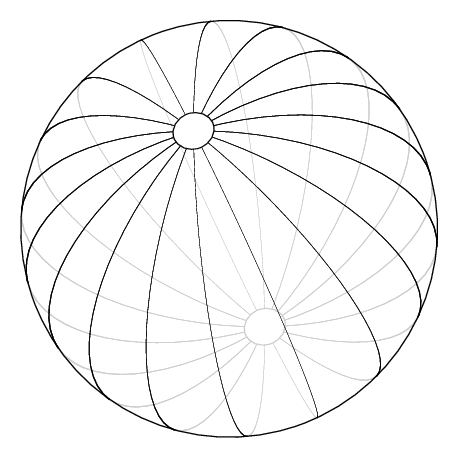}
\caption{A beachball of the first kind, with 18 complementary bigon regions.}\label{beachball_fig}
\end{figure}

Suppose $\Gamma_F$ is not a beachball.
The {\em reduced graph} $\overline{\Gamma}_F$ associated to $\Gamma_F$
is obtained from $\Gamma_F$ by collapsing all bigon regions. More generally, a reduced graph $\overline{\Gamma}$
in a surface $\widehat{F}$ is any graph with no complementary monogon or bigon regions.

\begin{lemma}\label{reduced_graph_euler_count}
Let $\overline{\Gamma}$ be a reduced graph in $\widehat{F}$ with $\overline{e}$ edges.
Then $-\chi(F) \ge \overline{e}/3$.
\end{lemma}
\begin{proof}
Let $\overline{v}$ be the number of vertices of $\overline{\Gamma}$, and $\overline{f}$ the number of complementary disk faces.
Non-disk faces contribute non-positively to Euler characteristic, so 
$\chi(\widehat{F}) \le \overline{v} - \overline{e} + \overline{f}$. Hence 
$\chi(F) = \chi(\widehat{F}) - \overline{v} \le \overline{f} - \overline{e}$. Since $\overline{\Gamma}$ is reduced, it
has no monogon or bigon faces, so $2\overline{e} \ge 3 \overline{f}$. Hence $3\chi(F) \le 3\overline{f} - 3\overline{e} \le -\overline{e}$.
\end{proof}

Edges in $\Gamma_F$ that cobound a bigon are said to be {\em parallel}. If $\Gamma_F$ is complicated, either
$-\chi(F)$ is large by Lemma~\ref{reduced_graph_euler_count}, or else there are many parallel edges. The next lemma
discusses the latter possibility. But first we introduce some terminology.

\begin{notation}
An interior edge of $\Gamma_G$ that joins vertices with index labels $i$ and $j$ will be called an {\em $(i,j)$-edge}.
\end{notation}

\begin{lemma}\label{parallel_edges}
If $\Gamma_F$ contains $(mn+1)$ parallel interior edges where $m \ge 1$, then there exists an index $k$ such that
\begin{enumerate}
\item{for each index $i$, the graph $\Gamma_G$ has $2m$ edges which are $(i,k-i)$-edges; and}
\item{for some index $i_0$, the graph $\Gamma_G$ has $(2m+1)$ edges which are $(i_0,k-i_0)$-edges.}
\end{enumerate}
\end{lemma}
\begin{proof}
Since all boundary components of $F$ are oriented coherently on $K$ (one says the vertices of $\Gamma_F$ have the
same {\em sign}), there is some (odd) index $k$ such that the index labels of any edge in the family are $i$ and $k-i$
(taken mod $n$). Since the family contains $(mn+1)$ edges, there is an index label $i_0$ that appears $(m+1)$ times at one
end of the family, and $m$ times at the other end; this proves the second claim. Moreover, any index label $i$ 
appears at least $m$ times at each end, proving the first claim.
\end{proof}

Every edge of $\Gamma_G$ is an arc of intersection of $F$ with $G$, and therefore corresponds to an edge of
$\Gamma_F$, and conversely. The next two lemmas control what happens when there are pairs of edges that are parallel
on both graphs simultaneously.

\begin{lemma}\label{interior_parallel}
Suppose there are interior edges that are parallel on both $\Gamma_F$ and $\Gamma_G$. Then 
$(M,K) = (M',K') \# (\RP^3,\RP^1)$.
\end{lemma}
\begin{proof}
Since all vertices of $\Gamma_F$ have the same sign, this follows from the argument in \cite{Gordon_Litherland}, proof
of Proposition~1.3. We observe that this argument is still valid if $F \cap G$ has simple closed curve components.
\end{proof}

\begin{lemma}\label{boundary_parallel}
Suppose there are boundary edges that are parallel on both $\Gamma_F$ and $\Gamma_G$. Then $K$ is
isotopic into $\partial M$.
\end{lemma}
\begin{proof}
By \cite{CGLS}, Lemma~2.5.4, such a pair of boundary edges gives an essential annulus $A$ in $X$ with one boundary
component on $\partial M$ and one on $\partial N(K)$, the latter having intersection number $1$ with the meridian of
$K$. Again this argument is valid in the presence of simple closed curve components of $F \cap G$.
This annulus can be used to define an isotopy of $K$ into $\partial M$.
\end{proof}

\subsection{Cables, satellites and tori}

In the sequel, many arguments will depend on relativizing to a knot in a simple $3$-manifold with boundary
(i.e.\ a submanifold of $M$). In this section, we analyze the most important special cases.

\begin{definition}
Let $K$ be a knot in $M$, with regular neighborhood $N(K)$. Let $K'$ be a simple closed curve on $\partial N(K)$ that
is essential in $N(K)$. Then we call $K'$ a {\em cable} of $K$.

A knot $K'$ contained in $N(K)$ is called a {\em satellite} of $K$ if it is not
contained in a $3$-ball in $N(K)$. If $[K']=k[K] \in H_1(N(K))$ then $k$ is called
the {\em winding number} of the satellite.
\end{definition}

\begin{remark}
Note that our definitions of a satellite and of a cable include the trivial cases where $K'$ is
isotopic to $K$ in $N(K)$.
\end{remark}

\begin{proposition}\label{satellite_proposition}
Let $K_0$ be a knot in a $3$-manifold $M$ whose exterior has incompressible boundary, and let $K$ be a satellite of $K_0$
with winding number $k>0$. Then $\|K\|\ge k\|K_0\|$.
\end{proposition}
\begin{proof}
By the definition of satellite, $K$ is contained in a solid torus $V$ in $M$
whose core is $K_0$.
Let $X_0 = M- \inte V$. 
Let $S$ be a good $p$-Seifert surface for $K$, $F = S\cap (V - \inte N(K))$, and $S_0 = S\cap X_0$.

If $K_0$ is $p$-trivial for some $p$ the result is obvious. 
So we may assume that $\partial V$ is incompressible in $X_0$ 
(see Remark~\ref{remark_200}). 
Thus $\partial V$ is incompressible in $M-K$, and we may therefore assume 
that no component of $S_0$ or $F$ is a disk. 
Hence $\eta (S) = \eta (F) + \eta (S_0)$.

In $H_1(V)$, there is equality $[\partial S_0] = [\partial S] = p[K] = pk[K_0]$. Therefore 
$$\|K_0\| \le \eta(S_0)/pk \le \eta(S)/pk$$
Since $S$ can be chosen so that $\eta(S)/p$ is arbitrarily close to $\|K\|$, the result follows
(or one can just apply Proposition~\ref{rational_genus_is_achieved}).
\end{proof}

\begin{proposition}\label{compressible_torus_proposition}
Let $M$ be a $3$-manifold whose boundary contains a compressible torus $T$, 
and let $K$ be a knot in $M$ such that $T$
is incompressible in $M-K$. Let $F$ be a relative $p$-Seifert surface for $K$. Then either
\begin{enumerate}
\item{$-\chi^-(F) > p/6$; or}
\item{$K$ is isotopic into $T$.}
\end{enumerate}
\end{proposition}
\begin{proof}
It is enough to prove the proposition under the assumption that $F$ is good.

Let $D$ be a compressing disk for $\partial M$ in $M$, such that $D \cap N(K)$ consists of $n$ meridian disks
of $N(K)$, with $n$ minimal. By hypothesis, $n>0$. Let $X = M - \inte N(K)$ as usual.

Let $P = D \cap X$, a planar surface. By the minimality of $n$, the surface $P$ is incompressible in $X$. Let
$\Gamma_F$ and $\Gamma_P$ be the intersection graphs in $\widehat{F}$ and $D$ respectively. There are two cases
to consider.

{\noindent \bf Case A.} ($n=1$) In this case, $P$ is an annulus, $\Gamma_F$ and $\Gamma_P$ have $p$ edges, and all
edges are boundary edges. In particular, all edges of $\Gamma_P$ are parallel. If $\Gamma_F$ has a pair of parallel
edges then $K$ is isotopic into $T$ by the proof of Lemma~\ref{boundary_parallel}. If not, then either $\Gamma_F$ is a
beachball of the second kind with a single edge, or the reduced graph $\overline{\Gamma}_F$ is defined and is
equal to $\Gamma_F$. In the first case the annulus $F$ defines an isotopy of $K$ into
$T$. In the second case $-\chi(F) \ge p/3$ by Lemma~\ref{reduced_graph_euler_count}.

\medskip

{\noindent \bf Case B.} ($n>1$) This case depends on two sublemmas.

\begin{sublemma}\label{parallel_interior_edges_sublemma}
The graph $\Gamma_F$ contains no family of $(\lfloor n/2 \rfloor +1)$ parallel interior edges.
\end{sublemma}
\begin{proof}
Since all the vertices of $\Gamma_F$ have the same sign, such a family would contain a (length $2$) 
{\em Scharlemann cycle} (hereafter referred to as an {\em $S$-cycle}) 
i.e.\ a configuration of the form depicted in Figure~\ref{S_cycle}.
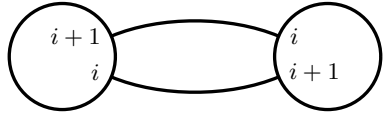
\begin{figure}[thbp]
\labellist
\small\hair 2pt
\pinlabel $i+1$ at 110 215
\pinlabel $i$ at 125 188
\pinlabel $i$ at 275 215
\pinlabel $i+1$ at 290 188
\endlabellist
\centering
\includegraphics[scale=0.5]{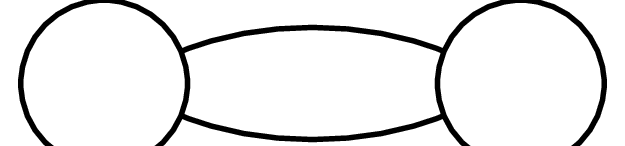}
\caption{An $S$-cycle}\label{S_cycle}
\end{figure}
As in \cite{Scharlemann_smooth}, Prop.~4.7 this $S$-cycle can be used to tube and compress $D$, giving a new
compressing disk $D'$ with $\partial D' = \partial D$ and $|D \cap K| = n-2$, contradicting minimality of $n$.
\end{proof}

\begin{sublemma}\label{parallel_boundary_edges_sublemma}
If $\Gamma_F$ contains $(2n-1)$ parallel boundary edges then $K$ is isotopic into $T$.
\end{sublemma}
\begin{proof}
In a family of $(2n-1)$ parallel boundary edges, the labels (on the vertex at one end of the family) cycle through
a full set of labels twice, with exactly one exception. Hence in $\Gamma_P$ we get a pair of boundary
edges at each vertex except (at most) one. Since $D$ is a disk, a pair of boundary edges together with a common vertex
separates $D$, so there is an outermost pair with the property that one of the complementary regions contains no
other vertex of $\Gamma_P$. But this means that the outermost pair of boundary edges are in fact parallel, so we obtain
a pair of boundary edges that are parallel in both $\Gamma_F$ and $\Gamma_P$. The desired result now follows
from the proof of Lemma~\ref{boundary_parallel}.
\end{proof}

We now complete the proof in Case B. First note that Sublemma~\ref{parallel_interior_edges_sublemma} 
implies that $\Gamma_F$ is not a beachball of the first kind. 
If $\Gamma_F$ is a beachball of the second kind then by Sublemma~\ref{parallel_boundary_edges_sublemma}
we may assume that $p=1$. 
But then the boundary component of $F$ that lies on $\partial N(K)$ 
intersects the meridian of $K$ exactly once, and so $F$ defines an isotopy 
of $K$ into $T$. 
We may therefore suppose that the reduced graph $\bar \Gamma_F$ of 
$\Gamma_F$ on $\widehat F$ exists, with, say, $\bar e_i$ interior edges 
and $\bar e_\partial$ boundary edges. By Sublemma~\ref{parallel_interior_edges_sublemma}
each interior edge of $\overline{\Gamma}_F$ corresponds to at most $\lfloor n/2 \rfloor$ edges of $\Gamma_F$,
and by Sublemma~\ref{parallel_boundary_edges_sublemma} we may assume that each boundary edge of $\overline{\Gamma}_F$
corresponds to at most $(2n-2)$ edges of $\Gamma_F$. Thus edges of $\overline{\Gamma}_F$ contribute at most $n$ and
$(2n-2)$ to the sum of valences at all vertices of $\Gamma_F$ respectively. Since this total sum is $pn$, and since
$n \le 2n-2$, we get $\overline{e} = \overline{e}_i + \overline{e}_\partial \ge pn/(2n-2) > p/2$. The conclusion now
follows from Lemma~\ref{reduced_graph_euler_count}.
\end{proof}

\begin{corollary}\label{double_cable_corollary}
Let $K_0$ be a knot in a $3$-manifold $M$ that is not $m$-trivial for 
any $m$, and let $K$ be a non-trivial cable of a non-trivial cable of $K_0$.
Then $\|K\| > 1/12$.
\end{corollary}
\begin{proof}
Let $V$ be a regular  neighborhood of $K_0$ containing $K$, and let 
$X_0 = M-\inte V$. 
By the hypothesis on $K_0$, $\partial V$ is incompressible in $X_0$. 
Let $S$ be a good $p$-Seifert surface for $K$ in $M$, such that 
$\|K\| = \eta (S)/p$, and let $F= S\cap V$ and $S_0 = S\cap X_0$. 
Then $F$ is a relative $p$-Seifert surface for $K$ in $V$.
Since $\partial V$ is incompressible in $M-K$, we may assume that no component 
of $F$ or $S_0$ is a disk, and hence $\eta (S) = \eta (F) + \eta(S_0)$, 
giving $\eta (F) \le \eta (S)$.  

Since $K$ is not a cable of $K_0$, Proposition~\ref{compressible_torus_proposition}, applied to 
$K\subset V$, implies that $-\chi (F)>p/6$.
Therefore $\|K\| = \eta (S)/p\ge \eta (F)/p >1/12$.
\end{proof}

The next proposition considers knots contained in a neighborhood of a torus.

\begin{proposition}\label{torus_product_proposition}
Let $K$ be a knot in $T^2 \times I$, and let $F$ be a relative $p$-Seifert surface for $K$. Then either
\begin{enumerate}
\item{$-\chi^-(F) \ge p/3$; or}
\item{$K$ is isotopic into $T^2 \times \lbrace 1/2\rbrace$.}
\end{enumerate}
\end{proposition}
\begin{proof}
We assume $F$ is good.

Put $K$ in thin position with respect to the torus $\widehat{T}:=T^2 \times \lbrace 1/2 \rbrace$. 
Let $n = |K \cap \widehat{T}|$, an even integer $\ge 0$. If $n=0$ the second conclusion holds, so assume $n\ge 2$.
Construct $\Gamma_F,\Gamma_T$ without monogons, as in Lemma~\ref{no_monogons}. Note that all edges
are interior edges. We require a sublemma.

\begin{sublemma}\label{parallel_edges_sublemma}
$\Gamma_F$ does not contains a family of $(n/2 + 1)$ parallel edges.
\end{sublemma}
\begin{proof}
Any such family contains an $S$-cycle. This gives rise to a M\"obius band properly embedded in (say) $T^2 \times [1/2,1]$,
which is absurd.
\end{proof}

By Sublemma~\ref{parallel_edges_sublemma}, we obtain an estimate $\overline{e} \ge e/(n/2) = p$, where $\overline{e}$ is
the number of edges in the reduced graph $\overline{\Gamma}_F$. (Note that $\overline{\Gamma}_F$ exists, since
otherwise $\Gamma_F$ is a beachball of the first kind and $K$ satisfies part~(3) of Theorem~\ref{vanishing_theorem}.
But this would give a Klein bottle embedded in $T^2 \times I$, which is absurd.)
Now apply Lemma~\ref{reduced_graph_euler_count}.
\end{proof}

As a corollary, we deduce the following:

\begin{corollary}\label{torus_corollary}
Let $T$ be an incompressible torus in a $3$-manifold $M$, and let $K$ be a knot in $M$ that lies in a regular
neighborhood of $T$. Let $F$ be a relative $p$-Seifert surface for $K$ in $M$. Then either
\begin{enumerate}
\item{$-\chi^-(F) \ge p/3$; or}
\item{$K$ is isotopic into $T$.}
\end{enumerate}
\end{corollary}
\begin{proof}
Let $N$ be a regular neighborhood of $T$, and define $F' = F \cap N$, $F'' = F \cap (\overline{M-N})$. Since
$T$ is incompressible in $M$ we may assume that $F''$ has no disk (or sphere) components. Therefore
$-\chi^-(F) \ge -\chi^-(F')$, and the result follows from Proposition~\ref{torus_product_proposition}.
\end{proof}

\section{Hyperbolic knots}\label{hyperbolic_knots_section}

In this section we consider the case that $M-K$ is hyperbolic; 
i.e.\ that $K$ is a {\em hyperbolic knot}. The arguments in this section use more geometry and analysis.

\subsection{Stable commutator length}\label{stable_commutator_length_subsection}

\begin{definition}\label{scl_definition}
Let $G$ be a group, and $a$ an element in $[G,G]$. The 
{\em commutator length} of $a$, denoted $\cl(a)$, is the minimal
number of commutators in $G$ whose product is $a$, and the
{\em stable commutator length}, denoted $\scl(a)$, is the limit
$$\scl(a) = \liminf_{n \to \infty} \frac {\cl(a^n)} {n}$$
\end{definition}

From the definition one sees that $\cl$ (and therefore also $\scl$)
is a characteristic function, and therefore in particular it is constant
on conjugacy classes. The function $\scl$ can be extended to conjugacy
classes which represent torsion elements in $H_1(G;\Z)$ by the formula
$\scl(a) = \scl(a^n)/n$ for any positive integer $n$. 

For an introduction to stable commutator length and its properties,
see \cite{Bavard} or \cite{Calegari_scl}.

There is a straightforward relationship between (stable) commutator 
length and norm, as follows.

\begin{lemma}\label{scl_genus_inequality}
Let $M$ be a $3$-manifold, and $K \subset M$ a knot. Let
$a \in \pi_1(M)$ be an element in the conjugacy class determined by the
free homotopy class of $K$. Then $\scl(a) \le \|K\|$. 
\end{lemma}
\begin{proof}
Proposition~2.10 from \cite{Calegari_scl} says that $\scl(a) = \inf_S -\chi^-(S)/2n$
where the infimum is taken over all oriented surfaces $S$ mapping to $M$ with
boundary $\partial S$ mapping to $K$ with total degree $n$. If $S$ is a $p$-Seifert
surface for $K$, collapsing $N(K)$ to $K$ wraps $\partial S$ around $K$ with total
degree $p$. The result follows.
\end{proof}

In general, $\scl(a)$ can be smaller than $\|K\|$, since the infimum in the
geometric definition of $\scl(a)$ (in the proof of Lemma~\ref{scl_genus_inequality})
is taken over {\em all} surfaces $S$ in $M$ which bound $K$, and not just {\em embedded} surfaces
whose interior is disjoint from $K$.
In other words, $\scl(a)$ is the best lower bound on $\|K\|$ which can be estimated
from the {\em homotopy} class of $K$. If $\|K\|$ is small, then $\scl(a)$ is small
and we will deduce information about the homotopy class of $K$ from this.

\subsection{Stable commutator length in hyperbolic $3$-manifolds}

There are strong interactions between geometry and $\scl$, especially in
dimension $3$. The most interesting case is that of hyperbolic geometry,
summarized in the following theorem:

\begin{theorem}[\cite{Calegari_stable}, Theorem~C]\label{Theorem_C}
For every $\epsilon>0$ there is a constant $\delta(\epsilon)>0$ such
that, if $M$ is a complete hyperbolic $3$-manifold and a nontrivial $a \in \pi_1(M)$
has $\scl(a) \le \delta$, then either $a$ is parabolic, or otherwise
if $\gamma$ is the unique geodesic in
the free homotopy class associated to the conjugacy class of $a$,
$$\length(\gamma) \le \epsilon$$
\end{theorem}

The dependence of $\delta$ on $\epsilon$ is not proper: in every finite volume hyperbolic
$3$-manifold, conjugacy classes $a$ with
$\scl(a) = 1/2$ correspond to arbitrarily long geodesics. However,
if $K$ is a knot with {\em sufficiently} small norm
in a closed hyperbolic $3$-manifold, Theorem~\ref{Theorem_C} implies that
$K$ is homotopic to a power of the core
geodesic of a Margulis tube, and the length of the geodesic can be bounded
from above by a constant depending on $\|K\|$.

In more detail, recall that Margulis showed that in each dimension $n$, there is
a positive constant $\epsilon(n)$ so that every geodesic in a closed hyperbolic $n$-manifold
of length at most $\epsilon(n)$ is simple, and is contained in an embedded solid tube whose
diameter can be estimated from below by a function of length. The exact value of the biggest
constant $\epsilon(n)$ with this property is not known when $n>2$, so for the sake of
precision, we make the following definition.

\begin{definition}\label{Margulis_tube_definition}
A {\em Margulis tube} in a hyperbolic $3$-manifold is an embedded solid tube of radius at least $0.531$ 
around a simple geodesic (the {\em core} of the Margulis tube) of length at most $0.162286$.
\end{definition}

The precise choice of constants are somewhat arbitrary, but are chosen to be compatible with
the estimates obtained by Hodgson-Kerckhoff \cite{Hodgson_Kerckhoff}; see the discussion in
the next subsection.

\subsection{Deformation of cone-manifold structure}

We have seen in the previous subsection that if $K$ is a knot of sufficiently small norm,
$K$ is homotopic into the core of a Margulis tube. 

We would like to conclude in fact that $K$ is {\em isotopic} to (a cable of) the core of the tube.
To do this we must use the fact that $\|K\|$ is small, not just $\scl$ of the corresponding
conjugacy class in $\pi_1(M)$. To make use of this fact, we must study the geometry of
$M-K$. In what follows we make use of some well-known results from the theory of
hyperbolic cone manifolds (with non-singular cone locus). For a reference see e.g.\
\cite{Boileau_et_al} or \cite{Hodgson_Kerckhoff}.

We assume throughout this section that $M-K$ is hyperbolic. Then as Thurston already
showed (see e.g.\ \cite{Thurston_notes}), for all sufficiently small 
positive real numbers $\theta$, there exists a hyperbolic cone manifold $M_\theta$,
unique up to isometry, whose underlying manifold is homeomorphic to $M$, and whose
cone locus is a geodesic in the isotopy class of $K$ with cone angle equal to
$\theta$.

The manifolds $M_\theta$ can be deformed (by increasing $\theta$) 
until one of the following happens:
\begin{enumerate}
\item{The cone angle can be increased all the way to $2\pi$, and one obtains a
complete hyperbolic structure on $M$ for which $K$ is isotopic to a geodesic}
\item{The volume goes to $0$ (and either converges after rescaling to a Euclidean cone
manifold, or the injectivity radius goes to $0$ everywhere after rescaling to have a fixed
diameter)}
\item{The cone locus bumps into itself (this can only happen for $\theta > \pi$)}
\end{enumerate}
In the second or third case we say that the cone manifold structure becomes singular.
Under deformation, the length of the cone geodesic strictly increases. For each
$\theta$, let $l(\theta)$ denote the length of the cone geodesic isotopic to $K$ in
$M_\theta$, and let $R(\theta)$ be the radius of a maximal open embedded tube around the cone geodesic.

Following Hodgson-Kerckhoff \cite{Hodgson_Kerckhoff} we define
$$h(R):= 3.3957 \frac {\tanh(R)} {\cosh(2R)}$$
which is non-negative for positive $R$, is asymptotic to $0$ as $R$ goes to $0$ or to
$\infty$, and which has a single maximum value $\approx 1.019675$, achieved
at $r \approx 0.531$.

The following is proved in \S~5 of \cite{Hodgson_Kerckhoff}:
\begin{theorem}[Hodgson-Kerckhoff]
Let $h(R)$ be as above. Let $l(\theta)$ be the length of the singular geodesic
in $M_\theta$. Then $M_\theta$ can be deformed (by increasing $\theta$) either until $\theta=2\pi$,
or until $\theta \cdot l$ is equal to the maximum of $h(R)$ (which occurs at approximately $R = 0.531$
and is equal to approximately $h(0.531) = h_\text{max} = 1.019675$)
and for all smaller values of $\theta$, the radius of a maximal embedded tube about the cone geodesic
is at least $0.531$.
\end{theorem}

It follows that we can deform $M_\theta$ either all the way to $\theta=2\pi$ with $l \le 0.162286$ and
$R\ge 0.531$, in which case $K$ is isotopic to the core of a Margulis tube, or else we can deform $M_\theta$
until $\theta \cdot l = h(0.531) \approx 1.019675$ for some $\theta < 2\pi$. In the
second case we can estimate $l(\theta) \ge 1.019675/2\pi \approx 0.162286$ and
$R(\theta) \ge 0.531$. In the next few sections we will obtain a positive
lower bound on the rational genus of $K$ in the second case.

\subsection{$1$-forms from tubes}\label{form_construction_section}

Let $M_\theta$ be a cone manifold, with a cone geodesic $\gamma$ with length $l$
and tube radius $R$. Let $T$ denote an embedded tube around $\gamma$ whose radius is $R$, and let
$p:T \to \gamma$ denote radial projection. Let $\phi:\gamma \to \R/l\cdot \Z$ be a parameterization
of $\gamma$ so that $d\phi$ is the length form on $\gamma$. Pulling back the $1$-form
$d\phi$ by $p^*$ defines a $1$-form on all of $T$ which, by abuse of notation, we denote
$d\phi$. Let $r:T \to [0,R]$ be the function on $T$
which is equal to the radial distance to $\gamma$. Define a $1$-form $\alpha$ on $M_\theta$
by 
$$\alpha = d\phi\cdot(\sinh(R) - \sinh(r))$$
on $T$, and extend it by $0$ outside $T$. Let $\beta$ be a $C^\infty$ function on
$[0,R]$ taking the value $1$ in a neighborhood of $0$ and the value $0$
in a neighborhood of $R$, and satisfying $|\beta'|< 1/(R-\epsilon)$ throughout $[0,R]$,
for some small fixed $\epsilon$. Finally define $\alpha_\epsilon = \beta(r)\alpha$.
Then the form $\alpha_\epsilon$ is $C^\infty$ on $M-\gamma$, and satisfies
$\|d\alpha_\epsilon\| \le 1 + 1/(R-\epsilon)$ pointwise. Moreover, the integral of $\alpha_\epsilon$
along $\gamma$ is $l\cdot\sinh(R)$. For a proof of these estimates, see Lemma~4.3 from
\cite{Calegari_stable}.

\subsection{Wrapping}

Suppose $K$ is a knot of sufficiently small rational genus such that $M-K$ is hyperbolic.
If $M$ is hyperbolic, and $K$ is isotopic to an embedded geodesic, then $K$ is isotopic
to the core of a Margulis tube (whose length may be estimated from above in terms of $\|K\|$,
by Theorem~\ref{Theorem_C}). Otherwise, we can find a hyperbolic cone manifold structure
$M_\theta$ on the underlying topological manifold $M$, with cone angle $\theta < 2\pi$ along
a single cone geodesic in the isotopy class of $K$, whose length is at
least $0.162286$, and is contained in an embedded tube whose radius is at least $0.531$.
For each $\epsilon > 0$, let $\alpha_\epsilon$ be a $1$-form constructed as in 
\S~\ref{form_construction_section}. Let $S$ be a good
$p$-Seifert surface for $K$ in $M_\theta$ realizing $\|K\|$, and let $S''$ be another (possibly
immersed) surface, homotopic to $S$ rel. boundary, with interior disjoint from $K$. For each $\epsilon$,
$$p\cdot l\cdot \sinh(R) = \int_{\partial S''} \alpha_\epsilon = \int_{S''} d\alpha_\epsilon \le \area(S'')\cdot(1+1/(R-\epsilon))$$
Taking $\epsilon \to 0$ we obtain an estimate
\begin{equation}\label{area_formula}
l\cdot\sinh(R)\cdot R/(R+1) \le \area(S'')/p
\end{equation}

By the discussion above, we can estimate
$$0.03131 \approx 0.162286\cdot \sinh(0.531)\cdot 0.531/1.531 \le l\cdot \sinh(R)\cdot R/(R+1)$$
We claim that one can find a representative surface $S''$  homotopic to $S$ rel. boundary in
$M-K$, of area at most $\epsilon-2\pi\chi(S)$ for any $\epsilon>0$. 
This will imply that $\|K\| \ge 0.03131/4\pi \approx 2.491\times 10^{-3}$.

\medskip

The representative surface $S''$ is obtained by {\em wrapping}; there are two (essentially equivalent)
methods to construct a ``wrapped'' surface: the shrinkwrapping method from
\cite{Calegari_Gabai}, and the PL wrapping method from \cite{Soma}. For technical ease, we use the
PL wrapping method. Roughly speaking, given a surface $S$ in a hyperbolic $3$-manifold and a prescribed
family $\Gamma$ of geodesics, the PL wrapping technique finds a $\CAT(-1)$ representative in the
homotopy class of $S$, which can be approximated by surfaces homotopic to $S$ in the complement of
$\Gamma$. 

Our situation is analogous to, but not strictly equivalent to, the situation in \cite{Soma} or \cite{Calegari_Gabai}.
In our context $\Gamma$ will be a singular cone geodesic, in the isotopy class of $K$,
in a cone manifold structure on $M$; and the
surface $S$ will have boundary wrapping some number of times around $\Gamma$ but interior
disjoint from $\Gamma$. In fact, this extra complication does not add any difficulty to the argument, and
we will obtain the same conclusion --- namely, the existence of a $\CAT(-1)$ surface
which can be approximated by surfaces homotopic to $S$ (rel. boundary) in the complement of $\Gamma$.
For the sake of completeness, we explain the construction in detail.

The key point of the construction is that the universal cover of $M-\Gamma$ has a metric completion
which is intrinsically $\CAT(-1)$; this is Lemma~1.2 from \cite{Soma}. Let $N$ denote the universal
cover of $M-\Gamma$, and $\overline{N}$ its metric completion. Notice that $\overline{N}$ is obtained
from $N$ by adding geodesics which project to components of $\Gamma$.

Let $T$ be a triangulation of $S$ with all vertices on $\partial S$. After an isotopy, we can
assume that the map $\partial S \to \Gamma$ takes all vertices of the triangulation to distinct
points in $\Gamma$.

The interior of each edge $e$ 
of $T$ not on $\partial S$ lifts to an open interval in $N$ whose closure is a closed interval
$\til{e}$ in $\overline{N}$. Since $\overline{N}$ is $\CAT(-1)$, there is a unique geodesic
$\til{e}'$ in $\overline{N}$ with the same endpoints as $\til{e}$. This projects to a piecewise geodesic 
segment $e'$ in $M$ with vertices on $\Gamma$ (note that the projection $e'$ does not depend on the
choice of the lift $\til{e}$). For each triangle $\Delta$ of $T$, we can choose lifts $\til{e}_i$
of the edges $e_i$ which together span a triangle $\til{\Delta}$ in $\overline{N}$. After replacing
each $\til{e}_i$ with a geodesic $\til{e}'_i$, we straighten $\til{\Delta}$ to a piecewise linear
surface by coning one vertex to the points on the opposite edge; i.e. if $v$ is the vertex
opposite $\til{e}'_1$ (say), for each point $p$ on $\til{e}'_1$ there is a unique geodesic in
$\overline{N}$ from $v$ to $p$, and the union of these geodesics is a piecewise totally geodesic
disk $\til{\Delta}'$ spanning the union of the $\til{e}'_i$. Let $\Delta'$ denote the projection
of $\til{\Delta}'$. Then the union of the $\Delta'$ is a $\CAT(-1)$ surface $S'$ which can be approximated
by surfaces whose interior is disjoint from $\Gamma$, and which are homotopic to $S$ through surfaces
with interior disjoint from $\Gamma$. The approximating surfaces $S''$ can be chosen to have area
as close to the area of $S'$ as desired; since $S'$ is $\CAT(-1)$, by Gauss-Bonnet, we have
$\area(S'') \le \epsilon - 2\pi\chi(S)$, as claimed.

\begin{remark}
The argument in \cite{Soma} uses extra hypotheses on $S$, namely that it is incompressible and
$2$-incompressible rel. $\Gamma$. In fact, these hypotheses are superfluous in the case that $S$ 
has boundary. In fact, even when $S$ is a closed surface,
one really only needs to know that $S$ contains {\em some} embedded loop which is covered by a nondegenerate
infinite geodesic in $\overline{N}$.
\end{remark}

\begin{remark}
The reader familiar with the construction of pleated surfaces in hyperbolic or $\CAT(-1)$ spaces will
recognize the similarity with PL wrapping.
\end{remark}

In fact, for our applications, it is important to construct surfaces $S''$ as above when $M$ has
boundary consisting of tori, and $S$ is a relative $p$-Seifert surface for $K$. There are no extra difficulties
in this case. One can proceed by either of two methods. The easiest is to construct the PL wrapped
surface directly. The triangulation of the relative $p$-Seifert surface can include vertices on a
``cusp'' boundary component; one just needs to observe that semi-infinite rays in $\overline{N}$
have geodesic representatives (constructed e.g. by taking limits of sequences of geodesic arcs
with one endpoint going to infinity along the ray). Triangles with two vertices on a single
``cusp'' boundary component degenerate to a geodesic ray under straightening. The resulting surface,
while non-compact, is complete, and satisfies $\area \le -2\pi\chi$.

Alternately, one can deform the metric in a neighborhood of infinity to make it $\CAT(0)$, in such a
way that each end is foliated as a metric product by Euclidean totally geodesic tori; such a metric
is described explicitly in the proof of Lemma~7.12 in \cite{Calegari_Gabai}. One obtains a compact
PL wrapped surface as above with some boundary components on $\Gamma$, and some on a fixed 
family of Euclidean tori, one for each cusp component of $M$. The PL wrapped surface so obtained is
$\CAT(0)$, and its restriction to any prescribed compact region of $M$ can be taken to be
$\CAT(-1)$; in particular, we can assume that the surface is $\CAT(-1)$
in the support of the $2$-form $d\alpha_\epsilon$, so that the area of the part of the surface in
the support of $d\alpha_\epsilon$ is at most $-2\pi\chi(S)$, and we obtain the desired bound on $-\chi(S)$.

\medskip

We have therefore obtained a proof of the
following theorem:

\begin{theorem}\label{hyperbolic_complement_estimate}
Let $K$ be a knot in a $3$-manifold $M$, possibly with boundary consisting of a union of tori.
Suppose $M-K$ is hyperbolic. 
Then either 
\begin{enumerate}
\item{$\|K\| \ge 1/402$; or}
\item{$M$ is hyperbolic and $K$ is isotopic to the core of a Margulis tube.}
\end{enumerate}
Moreover, if $F$ is a relative $p$-Seifert surface for $K$ in $M$, then either
\begin{enumerate}
\item{$\eta(F)/p \ge 1/402$; or}
\item{$M$ is hyperbolic and $K$ is isotopic to the core of a Margulis tube.}
\end{enumerate}
\end{theorem}

\subsection{Better estimates}

In fact, though the wrapping technique explains in a direct geometric way
the relationship between $\|K\|$ and the geometry of $K$ in $M$, one can obtain better estimates at
the cost of appealing to some more refined technology of Agol and Cao--Meyerhoff (which we treat as
a black box). For the benefit of
the reader we include a sketch of a proof of the following:

\begin{proposition}\label{improved_hyperbolic_estimate}
Let $K$ be a knot in a $3$-manifold $M$. Suppose $M-K$ is hyperbolic. Then either
\begin{enumerate}
\item{$\|K\| \ge 1/50$; or}
\item{$M$ is hyperbolic and $K$ is isotopic to the core of a Margulis tube.}
\end{enumerate}
\end{proposition}
\begin{proof}
Let $S$ be a $p$-Seifert surface for $K$.
Let $T$ be a maximal horospherical cusp torus in $M-K$; we may regard $S$ as satisfying
$\partial S \subset T$. Since $S$ is essential, Thm.~5.1 from \cite{Agol} gives $|\chi(S)| \ge \ell(\partial S)/6$
where $\ell$ denotes Euclidean length measured on $T$. Note that $\ell(\partial S)=q\ell(\sigma)$, where
$q$ is the number of boundary components of $S$, and $\sigma$ is the boundary slope. Also,
$p=q\Delta(\sigma,\mu)$, where $\Delta(\sigma,\mu)$ is the geometric intersection number of $\sigma$ with the
meridian slope $\mu$. Hence
$$|\chi(S)|/2p \ge \ell(\partial S)/12p = q\ell(\sigma)/12q\Delta(\sigma,\mu) = \ell(\sigma)/12\Delta(\sigma,\mu)$$
Let $A$ be the area of $T$. Then by the proof of Thm.~8.1 from \cite{Agol},
$$A \le \ell(\sigma)\ell(\mu)/\Delta(\sigma,\mu)$$
Hence
$$|\chi(S)|/2p \ge A/12\ell(\mu) \ge \sqrt{A}/12\ell_N(\mu)$$
where $\ell_N(\mu):=\ell(\mu)/\sqrt{A}$ is the {\em normalized length} of $\mu$ (see \cite{Cao_Meyerhoff}).
By \cite{Cao_Meyerhoff} one knows without hypothesis that $A\ge 3.35$. If we define 
$$C:=\sqrt{3.35}/12(7.515) \sim 0.0203 > 1/50$$
then we conclude
$$|\chi(S)|/2p < C \text{ implies that } \ell_N(\mu)\ge 7.515$$
and therefore by \cite{Cao_Meyerhoff} p.~410 we deduce that $M$ is hyperbolic and $K$ is isotopic to the core
of a Margulis tube (i.e.\/ a geodesic of length $\le 0.162$ with tube radius $\ge 0.531$).
\end{proof}

It is therefore probably safe to replace $1/402$ by $1/50$ throughout the sequel by appealing to 
Proposition~\ref{improved_hyperbolic_estimate} in place of Theorem~\ref{hyperbolic_complement_estimate}
(it is, however, unlikely that $1/50$ is sharp).

\section{General knots}\label{general_knots_section}

We are now in a position to discuss the most general case of a knot $K$ in a
closed, orientable $3$-manifold $M$ such that $[K]$ has finite order in $H_1(M)$.
The discussion is case-by-case, and depends on the (well-known) prime and JSJ decomposition
theorems.

\medskip

The first step is to consider the interaction of $(M,K)$ with the essential $2$-spheres in $M$.
Such spheres are treated by the following theorem.

\begin{theorem}\label{reducible_theorem}
Let $K$ be a knot in a reducible manifold $M$. 
Then either 
\begin{enumerate}
\item{$\|K\| \ge 1/12$; or}
\item{there is a decomposition $M = M' \# M''$, $K \subset M'$ and either
\begin{enumerate}
\item{$M'$ is irreducible, or}
\item{$(M',K) = (\RP^3,\RP^1)\#(\RP^3,\RP^1)$}
\end{enumerate}
}
\end{enumerate}
\end{theorem}

Note that in case (2)(b), $\|K\|=0$.

\begin{proof}
Let $S$ be a good $p$-Seifert surface for $K$.

First assume that $M-K$ is irreducible. 
Let $\Sigma$ be an essential $2$-sphere in $M$, chosen so that $n=|\Sigma\cap K|$
is minimal. 
Since $[K]$ has finite order in $H_1(M)$, the algebraic intersection
number of $K$ and $\Sigma$ is zero; thus $n$ is even and $>0$. 

Let $P$ be the planar surface $\Sigma - \inte N(K)$. 
Let $\Gamma_S$, $\Gamma_P$ be the intersection graphs in $\widehat S$ and 
$\Sigma$ respectively, defined by the arc components of $S\cap P$. 

{\noindent \bf Case A.} ($n\ge 4$)
An {\em extended $S$-cycle} is
a configuration of the form depicted in Figure~\ref{extend_S_cycle}; i.e.\ a series of
four parallel edges, whose middle pair form an ordinary $S$-cycle.

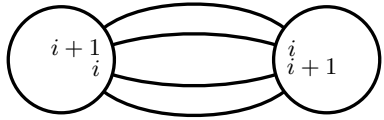
\begin{figure}[thbp]
\labellist
\small\hair 2pt
\pinlabel $i+1$ at 111 209
\pinlabel $i$ at 126 194
\pinlabel $i$ at 274 209
\pinlabel $i+1$ at 289 194
\endlabellist
\centering
\includegraphics[scale=0.5]{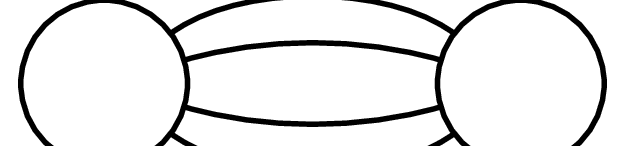}
\caption{An extended $S$-cycle}\label{extend_S_cycle}
\end{figure}

In this context, and with the assumption that $n\ge 4$, Lemma~2.3 from \cite{Wu}
says precisely that $\Gamma_S$ does not contain an extended $S$-cycle.
Hence $\Gamma_S$ does not contain a family of $(n/2 + 3)$ parallel edges.
Let $\overline{\Gamma}_S$ denote the reduced graph of $\Gamma_S$. Since $e=pn/2$, we can estimate
$\overline{e} \ge p(n/2)/(n/2 + 2) \ge p/2$ (because $n\ge 4$), where $\overline{e}$ denotes the number of edges
of $\overline{\Gamma}_S$. Hence by Lemma~\ref{reduced_graph_euler_count}, we have $-\chi(S)\ge p/6$ and
therefore $\|K\|\ge 1/12$.

\medskip

{\noindent \bf Case B.} ($n=2$)
Here $P$ is an annulus, so all edges of $\Gamma_P$ are parallel (i.e.\ $\Gamma_P$ is a beachball).

If $\Gamma_S$ has a pair of parallel edges, consider an innermost such pair 
$e_1,e_2$, i.e. $e_1$ and $e_2$ cobound a bigon face of $\Gamma_S$. 
The corresponding edges of $\Gamma_P$ are necessarily parallel, and it 
follows that their endpoints, $x_1$ and $x_2$, say, at a given vertex of 
$\Gamma_P$ are diametrically opposite, in the sense that there are the same
number of edge-endpoints in each of the two intervals around the vertex on 
either side of the pair $\{x_1,x_2\}$; see Figure~4 in \cite{Gordon_Litherland}. 
Therefore $\Gamma_S$ cannot have three mutually parallel edges. 

If $\Gamma_S$ is a beachball with exactly two edges, the argument in the proof
of Proposition~1.3, case (1) from \cite{Gordon_Litherland} shows that $\Sigma$ decomposes 
$(M,K)$ as a connected sum $(\RP^3,\RP^1)\,\#\, (\RP^3, \RP^1)$, 
which we think of as a degenerate case of case (2)(b) in 
the statement of the theorem. 

We may therefore assume that $\Gamma_S$ is not a beachball, and that 
no family of parallel edges in $\Gamma_S$ has more than two edges. 
Consequently the reduced graph $\bar \Gamma_S$ exists and its number 
of edges $\bar e\ge p/2$, so Lemma~\ref{reduced_graph_euler_count} gives $\chi(S) \le -p/6$, and 
hence $\|K\| \ge 1/12$. 
This completes the analysis when $M-K$ is irreducible.   

\medskip

If $M-K$ is reducible, we can write $M= M'\# M''$ where $K\subset M'$ 
and $M'-K$ is irreducible. 
Let $\Sigma'$ be a $2$-sphere in $M-K$ realizing the connected sum decomposition
${M-K = (M'-K)\,\#\, M''}$.
By surgering $S$ along the curves of intersection $S\cap \Sigma'$ we get a 
$p$-Seifert surface $S'$ for $K$ in $M'$, with $-\chi^-(S')\le -\chi^-(S)$.
The theorem now follows from the argument in the case that $M-K$ is 
irreducible.
\end{proof}

After Theorem~\ref{reducible_theorem} we may assume that $M$ is irreducible.
Moreover, if $M-K$ is reducible then $K$ lies in a 3-ball in $M$ and so either
$\|K\|\ge 1/2$ or $K$ bounds a disk (in which case $\|K\|=0$).
We therefore assume in the sequel that both $M$ and $M-K$ are irreducible.

\medskip

A closed, irreducible $3$-manifold $M$ is either $S^3$, a lens
space, an atoroidal Seifert fiber space over $S^2$ with three exceptional 
fibers, hyperbolic, or toroidal. The next theorem treats the case that $M$ is a lens space.

\begin{theorem}\label{lens_theorem} 
Let $K$ be a knot in a lens space $M$. 
Then either 
\begin{enumerate}
\item{$\|K\| \ge 1/24$; or}
\item{$K$ lies on a Heegaard torus in $M$; or}
\item{$M$ is of the form $L(4k,2k-1)$ and $K$ lies on a Klein bottle in $M$ as a non-separating
orientation-preserving curve.}
\end{enumerate}
\end{theorem}

Note that in case (3) we have $\|K\|=0$; this is a special case of Theorem~\ref{vanishing_theorem} (3). 

\begin{proof} 
Let $\widehat T$ be a Heegaard torus in $M$. 
Then either (2) holds, or we can put $K$ in thin position with respect 
to $\widehat T$. 
Assuming the latter, let $n = |K\cap \widehat T|$; so $n$ is even and $\ge 2$.
Let $S$ be a good $p$-Seifert surface for $K$ in $M$. 
We then get intersection graphs $\Gamma_S$, $\Gamma_T$. 
Thin position of $K$ and incompressibility of $S$ guarantee by Lemma~\ref{no_monogons}
that $\Gamma_S$ and $\Gamma_T$ have no monogon disk faces.

We need a sublemma.

\begin{sublemma}\label{lem9}
Suppose there is a pair of edges that are parallel on both $\Gamma_S$ and 
$\Gamma_T$. 
Then $M = \RP^3$, and either
\begin{enumerate}
\item{$\|K\| \ge 3/4$; or}
\item{$(M,K) = (\RP^3, \RP^1)$.}
\end{enumerate}
\end{sublemma}

\begin{proof} 
By Lemma~\ref{interior_parallel} we have 
$(M,K) = (M',K')\,\#\, (\RP^3,\RP^1)$. 
Since $M$ is a lens space we must have $M'= S^3$. 
If $K'$ is trivial we have (2). 
Suppose $K'$ is non-trivial. 
Then $\|K'\|\ge 1/2$ and, since $\RP^1$ in $\RP^3$ is 2-trivial, 
Theorem~\ref{connect_sum_theorem} implies that $\|K\| = \|K'\| - \frac14 +\frac12 \ge \frac34$.
\end{proof}

Note that in case~(2) (i.e.\ $(M,K) = (\RP^3,\RP^1)$) we have $\|K\| =0$.

If $\Gamma_S$ has $(2n+1)$ parallel edges then by Lemma~\ref{parallel_edges} bullet~(2)  
there are indices $i_0$ and $k$ so that $\Gamma_T$ has $5$ edges labeled $(i_0,k-i_0)$.
Remember that this means that there are vertices $i_0$ and $k-i_0$ in $\Gamma_T$ that are joined
to each other by at least $5$ edges. On a torus, one can find at most four embedded pairwise
non-parallel arcs joining two points, that are disjoint except at their endpoints. To
see this, ``engulf'' one of the edges by a disk, and observe that it is equivalent
to show that there are at most three pairwise non-parallel essential embedded loops
that intersect each other in one point; this latter fact can be shown using intersection
number. Consequently we can deduce that two of the $(i_0,k-i_0)$-edges must be parallel
in $\widehat T$.

If they are parallel in $\Gamma_T$ (i.e.\ if they cobound an embedded complementary bigon)
then there is a pair of edges that are parallel in both graphs, and 
Sublemma~\ref{lem9} applies. If not, then the disk $D$ in $\widehat T$ 
realizing the parallelism 
of the two edges must contain vertices of $\Gamma_T$ in its interior.
By Lemma~\ref{parallel_edges} bullet~(1), these vertices come in pairs $i$, $k-i$,
and each pair are joined by $4$ edges. 
An easy innermost argument in the disk $D$ shows that, for some $i$, 
some pair of $(i,k-i)$ edges are parallel in $\Gamma_T$. 
Hence again we get edges that are parallel in both graphs. 

By Sublemma~\ref{lem9} we may therefore suppose that there is no family of 
$(2n+1)$ parallel edges in $\Gamma_S$. 
If $\Gamma_S$ is not a beachball then the reduced graph $\overline{\Gamma}_S$ exists and satisfies
$\bar e \ge e/2n = pn/4n = p/4$.
Hence $-\chi^- (S) \ge p/12$ by Lemma~\ref{reduced_graph_euler_count}, giving 
$\|K\|\ge 1/24$.

If $\Gamma_S$ is a beachball then as observed in Remark~\ref{beachball_remark}, $K$ satisfies 
conclusion~(3) of Theorem~\ref{vanishing_theorem}, i.e. $K$ is the fiber of multiplicity $r$ 
in a Seifert fiber subspace $N$ of $M$ whose base orbifold is a M\"obius 
band with one orbifold point, of order $r\ge 1$. 
Since $M$ is a lens space, $\overline{M-N}$ is a solid torus, and the 
meridian of this solid torus is not the Seifert fiber of $N$. 
Hence $M$ is a Seifert fiber space with  orbifold $\RP^2$ and one 
orbifold point of order $k\ge 1$, where $k=r$ if $r>1$. 
Thus $M$ is the lens space $L(4k,2k-1)$. 
If $r>1$ then $K$ lies on a Heegaard torus in $M$. 
If $r=1$ then $K$ satisfies conclusion~(3).
\end{proof}

\begin{remark}
The examples in Case~B of \S~\ref{example_subsection} show that for any $\epsilon>0$ there exists 
a lens space $M$ and a knot $K$ lying on a Heegaard torus in $M$ with 
$0< \|K\| <\epsilon$.
\end{remark}

\begin{remark}
The bound $1/24$ in Theorem~\ref{lens_theorem} is almost certainly not best possible. 
The smallest value of $\|K\|$ that we know  of for a knot $K$ in a lens 
space not satisfying (2) or (3) of Theorem~\ref{lens_theorem}, comes from the following 
example. 

Let $K'$ be the $(-2,3,7)$-pretzel knot in $S^3$. 
Then 19-surgery on $K'$ gives a knot $K$ in the lens space $L(19,7)$.
By Lemma~\ref{surgery_lemma}, $\|K\| = (g(K')-1/2)/19 = (5-\, 1/2)/19 = 9/38$.
Since $M- K\cong S^3 -K'$ is hyperbolic, $K$ does not lie on a Heegaard 
torus in $M$. 
In fact $K'$ is a Berge knot, and so $K$ is 1-bridge in $M$.
\end{remark}

\begin{remark}
Baker has shown in \cite{Baker} that if $K$ is a knot in a lens space $M$ 
such that $\|K\|$ is realized by a $p$-Seifert surface $S$ with a single 
boundary component then either $K$ is 1-bridge in $M$ or $\|K\| >1/4$.
By the proof of Lemma~\ref{surgery_lemma},
the hypothesis holds for pairs $(M,K)$ that come from surgery on a knot 
in $S^3$ (or any homology sphere).
\end{remark}

\begin{remark}
In our proof of Theorem~\ref{lens_theorem}, the argument shows that if $n>2$ then 
$\Gamma_S$ cannot have $(n+1)$ parallel edges. 
Hence (1) can be improved to $\|K\| \ge 1/12$ if (2) is weakened to 
say that $K$ is $0$- or $1$-bridge in $M$.
\end{remark}

The next proposition considers knots in hyperbolic $3$-manifolds, possibly with boundary.
The case that the complement of the knot is hyperbolic was already
treated in Theorem~\ref{hyperbolic_complement_estimate}. Here we treat the general case.

\begin{proposition}\label{hyperbolic_proposition}
Let $M$ be a hyperbolic $3$-manifold, possibly with boundary consisting 
of a union of tori. 
Let $K$ be a knot in $M$ such that $M-K$ is irreducible, and let $F$ 
be a relative $p$-Seifert surface for $K$.
Then either 
\begin{enumerate}
\item{$\eta (F)/p \ge 1/402$; or}
\item{$K$ is isotopic to a cable of a core of a Margulis tube; or}
\item{$K$ is isotopic into $\partial M$.}
\end{enumerate}
\end{proposition}

\begin{corollary}\label{hyperbolic_corollary}
Let $K$ be a knot in a closed hyperbolic $3$-manifold $M$.
Then either 
\begin{enumerate}
\item{$\|K\| \ge 1/402$; or}
\item{$K$ is trivial; or} 
\item{$K$ is isotopic to a cable of the core of a Margulis tube.}
\end{enumerate}
\end{corollary}

\begin{proof}[Proof of Proposition~\ref{hyperbolic_proposition}]
If $M-K$ is hyperbolic then the result follows from Theorem~\ref{hyperbolic_complement_estimate}.
So we may assume that $M-K$ is toroidal.

Since $M$ is hyperbolic, every essential torus in $M-K$ is separating, so 
there exists an {\em extremal} such torus, i.e.\ an essential torus $T$
in $M-K$ such that the component $X_0$ of $M$ cut along $T$ that does not 
contain $K$ is atoroidal. 
Let $N= \overline{M-X_0}$; thus $K$ is contained in $N$.

Let $F$ be a good relative $p$-Seifert surface for $K$ in $M$, and let 
$F' = F\cap N$, $F_0 = F\cap X_0$. 
Since $T$ is incompressible in $M-K$, and $M-K$ is irreducible, we may 
assume that no component of $F'$ or $F_0$ is a disk.
Hence $\eta (F) = \eta (F') + \eta (F_0)$.

Since $M$ is atoroidal, either $T$ compresses in $N$ or $N$ is a product 
$T^2 \times I$ where $T^2\times \{0\}$ is a component of $\partial M$ 
and $T^2 \times \{1\}=T$. 
In the latter case, Proposition~\ref{torus_product_proposition} implies that either $\eta (F) \ge 
\eta (F') \ge p/6$ or $K$ is isotopic into $\partial M$.

We may therefore assume that $T$ compresses in $N$. 
By Proposition~\ref{compressible_torus_proposition} either $\eta (F')/p>1/12$ or $K$ is isotopic into $T$.
So we may suppose the latter holds. 

First suppose that $N$ is a solid torus. 
Let $K_0$ be the core of $N$. 
Then $K$ is a non-trivial cable of $K_0$. 
The exterior of $K_0$ in $M$ is $X_0$, which is atoroidal.
Since $M$ is hyperbolic $X_0$ cannot be Seifert fibered, and therefore 
$X_0$ is hyperbolic.
By Proposition~\ref{satellite_proposition} $\|K\| \ge 2\|K_0\|$. Hence either we have
conclusion (1) or $|K_0| < 1/804$, in which case $K_0$ is isotopic to the core of a Margulis
tube by Theorem~\ref{hyperbolic_complement_estimate}.

If $N$ is not a solid torus then $X_0$ lies in a 3-ball in $M$.
Since $K$ is isotopic into $T= \partial X_0$, $K$ also lies in a 3-ball,
contradicting our assumption that $M-K$ 
is irreducible.
\end{proof}

It remains to consider toroidal manifolds, and small Seifert fiber spaces.
The small Seifert fiber spaces are treated by the following theorem.
Recall that a {\em prism manifold} is a Seifert fiber space $M$ with base $S^2$ and three
exceptional fibers of multiplicities $2,2,n$. Then $M$ has another Seifert fiber structure
with base $\RP^2$ and at most one exceptional fiber.

\begin{theorem}\label{small_sfs_theorem}
Let $M$ be an atoroidal Seifert fiber space over $S^2$ with three 
exceptional fibers and let $K$ be a knot in $M$. 
Then either
\begin{enumerate}
\item{$\|K\| \ge 1/402$; or}
\item{$K$ is trivial; or}
\item{$K$ is a cable of an exceptional Seifert fiber of $M$; or}
\item{$M$ is a prism manifold and $K$ is a fiber in the Seifert fiber structure of $M$ over $\RP^2$
with at most one exceptional fiber.}
\end{enumerate}
\end{theorem}

Note that in case~(4) we have $\|K\|=0$. This is a special case of Theorem~\ref{vanishing_theorem} (3).

\begin{proof}
We may assume that $X$ (i.e.\/ $M - \inte N(K)$) is irreducible.

If $X$ is hyperbolic then (1) holds by Theorem~\ref{hyperbolic_complement_estimate}.

If $X$ is Seifert fibered then (since $M$ is irreducible) the Seifert 
fibration of $X$ extends to $M$.
Thus $K$ is a Seifert fiber of this fibration. 
If $M$ is not a prism manifold then the Seifert fiber structure on $M$ is unique and (since an ordinary
fiber is a cable of an exceptional fiber) we have conclusion~(3). If $M$ is a prism manifold then
the two possible Seifert fiber structures on $M$ give either (3) or (4).

We may therefore assume that $X$ is toroidal. 
Since $M$ is irreducible and atoroidal, every torus in $X$ is separating. 
As in the proof of Proposition~\ref{hyperbolic_proposition}, let $T$ be an extremal essential torus 
in $X$, so $M = X_0 \cup_T N$ where $X_0$ is atoroidal. 
Again as in the proof of Proposition~\ref{hyperbolic_proposition}, we may assume that $K$ is isotopic 
into $T$, and that $N$ is a solid torus, with core $K_0$. 
By Proposition~\ref{satellite_proposition}, $\|K\| \ge 2\|K_0\|$. 
If $M-K_0$ is hyperbolic, and (1) does not hold, then $\|K_0\| <1/804$ and 
so $M$ is hyperbolic by Theorem~\ref{hyperbolic_complement_estimate}, a contradiction. 

Hence we may assume that $X_0$ is Seifert fibered. 
The Seifert fibration of $X_0$ extends to $M$, so $K_0$ is a fiber in the 
fibration of $M$, and $K$ is a non-trivial cable of $K_0$.
Since an ordinary fiber is a non-trivial cable of an exceptional fiber, 
by Corollary~\ref{double_cable_corollary} either $\|K\| >1/12$ or $K_0$ is an exceptional fiber.
\end{proof}

The toroidal case involves several subcases, which are treated in the next few propositions,
culminating in Theorem~\ref{toroidal_theorem}.

\begin{proposition}\label{toroidal_atoroidal_proposition}
Let $M$ be an irreducible, toroidal $3$-manifold, and let $K$ be 
a knot in $M$ such that $M-K$ is irreducible and atoroidal. 
Let $F$ be a relative $p$-Seifert surface for $K$. 
Then $-\chi^- (F) \ge p/12$.
\end{proposition}

\begin{proof}
We may suppose that $F$ is good.

Let $\widehat T$ be an essential torus in $M$ such that 
$n= |K\cap \widehat T|$ is minimal (over all essential tori in $M$).
Note that $n>0$ since $M-K$ is atoroidal.
Also, the existence of $F$ shows that $p[K] =0\in H_1(M,\partial M)$, 
which implies that $n$ is even. 

Let $T$ be the punctured torus $\widehat T\cap (M- \inte N(K))$ and let 
$\Gamma_F$, $\Gamma_T$ be the intersection graphs in $\widehat F$ and 
$\widehat T$ respectively. 
Note that all edges are interior edges.

We need some sublemmas.

\begin{sublemma}\label{lem10}
If $n\ge 4$ then $\Gamma_F$ contains no family of more than 
$(n/2+2)$ parallel edges. 
\end{sublemma}

\begin{proof} 
Since all vertices of $\Gamma_F$ are of the same sign, a family of more
than $(n/2 +2)$ parallel edges would contain an extended $S$-cycle.
But this is impossible, again by Lemma~2.3 from \cite{Wu} (compare with Case~A in the proof of Theorem~\ref{reducible_theorem}).
\end{proof}

\begin{sublemma}\label{lem11}
If $n=2$ then $\Gamma_F$ contains no family of $5$ parallel edges.
\end{sublemma}

\begin{proof} 
The condition $n=2$ implies that $\Gamma_T$ has $2$ vertices. Recall from the proof of
Theorem~\ref{lens_theorem} that on a torus one can find at most four embedded pairwise non-parallel
arcs joining two points, that are disjoint except at their endpoints.
Consequently if $n=2$ and $\Gamma_F$ contains a family of $5$ parallel edges, at least two of these
edges are also parallel in $\Gamma_T$, so Lemma~\ref{interior_parallel} implies that
$M=\RP^3$, contrary to the assumption that $M$ is toroidal.
\end{proof}

We now complete the proof of Proposition~\ref{toroidal_atoroidal_proposition}.
If $F$ is not an annulus, let $\overline{e}$ be the number of edges in the reduced graph $\overline{\Gamma}_F$.
If $n\ge 4$ then Sublemma~\ref{lem10} shows that $\overline{e}\ge pn/(n+4) \ge p/2$.
If $n=2$ Sublemma~\ref{lem11} shows that 
$\overline{e}\ge 2p/8 = p/4$.
Hence $-\chi^- (F)\ge p/12$ by 
Lemma~\ref{reduced_graph_euler_count}.

If $F$ is an annulus then by Theorem~\ref{vanishing_theorem} (3)
$K$ is contained as a fiber of multiplicity $r$ in a Seifert fiber 
submanifold $N$ of $M$, where $N$ has base orbifold a M\"obius band 
with one orbifold point of order $r\ge 1$. 
In particular $N-K$ is toroidal.
Since $M-K$ is atoroidal by hypothesis, $\partial N$ must be compressible 
in $W= \overline{M-N}$. 
Since $M-K$ is irreducible it follows that $W$ is a solid torus, and that 
the meridian of $W$ is not identified  with a fiber in $\partial N$. 
Hence $M- \inte N(K)$ is a Seifert fiber space over the M\"obius band
with at most one exceptional fiber (the core of $W$).
Since this manifold is atoroidal by hypothesis, 
the core of $W$ is in fact an ordinary fiber.
Then $M$ is a Seifert fiber space over $\RP^2$ with at most 
one exceptional fiber, namely $K$. 
But this contradicts our assumption that $M$ is toroidal. 
\end{proof}

\begin{proposition}\label{annular_proposition}
Let $M$ be an irreducible $3$-manifold containing an essential 
annulus $A$. Let $K$ be a knot in $M$ and let $F$ be a good relative
$p$-Seifert surface for $K$. Then either 
\begin{enumerate}
\item{$-\chi^- (F) \ge p/6$; or}
\item{$K$ can be isotoped to be either disjoint from $A$ or to lie 
in $\partial M$; or}
\item{$F$ is an annulus.}
\end{enumerate}
\end{proposition}

\begin{proof} 
We may assume that $F$ is good. 
Let $X= M - \inte N(K)$.

Isotop $K$ so that $n= |K\cap A|$ is minimal.
If $n=0$ we are done, so assume in what follows that $n>0$.
Let $P = A\cap X$, and let $\Gamma_F$, $\Gamma_P$ be the intersection 
graphs in $\widehat F$ and $A$ respectively.

\begin{sublemma}\label{lem12}
$\Gamma_F$ contains no family of $(n+1)$ parallel interior edges.
\end{sublemma}

\begin{proof}
By Lemma~\ref{parallel_edges} bullet~(2), $\Gamma_P$ has 3 $(i_0,k-i_0)$-edges for some 
$i_0,k$. Two of these must be parallel in $A$ (since the union of any two is an embedded loop, and there is
only one isotopy class of embedded essential loops in an annulus).
If they are parallel in $\Gamma_P$ then we have interior edges parallel 
in both graphs and Lemma~\ref{interior_parallel} readily gives a contradiction.
If not, then the disk $D$ in $A$ realizing the parallelism must have 
vertices of $\Gamma_P$ in its interior.
By Lemma~\ref{parallel_edges} bullet~(1) these vertices
come in pairs $i$, $k-i$, each with $2$ edges 
joining them. By taking an innermost such pair in $D$ we again get a pair of edges
that are parallel in both graphs, a contradiction as 
before.
\end{proof}

\begin{sublemma}\label{lem13} 
If $\Gamma_F$ contains a family of $(2n+1)$ parallel boundary edges 
then $K$ is isotopic into $\partial M$.
\end{sublemma}

\begin{proof}
At the vertex end of such a family every label in $\{1,2,\ldots,n\}$ 
appears twice, and one label appears three times.
The three corresponding boundary edges in $\Gamma_P$ all share a common vertex;
at least two must be parallel in $A$. If they are parallel in $\Gamma_P$, we get
boundary edges parallel in both graphs, so the result follows from Lemma~\ref{boundary_parallel}.
If not, then the disk $D$ in $A$ realizing the parallelism must have vertices
of $\Gamma_P$ in its interior. Each such vertex has two boundary edges coming from the given family;
an innermost pair is parallel in both graphs, so we are done by Lemma~\ref{boundary_parallel}.
\end{proof}

We now complete the proof of Proposition~\ref{annular_proposition}.
If $F$ is an annulus we are done. If not, then the reduced graph $\overline{\Gamma}_F$ exists.
Let $\overline{e}_i,\overline{e}_\partial$ be the number of interior and boundary 
edges of $\overline{\Gamma}_F$. 
By Sublemma~\ref{lem12} each interior edge of $\overline{\Gamma}_F$ contributes 
at most $2n$ to the sum of the valences of the vertices on $\Gamma_F$, 
and by Sublemma~\ref{lem13} we may assume that each boundary edge of 
$\overline{\Gamma}_F$ also contributes at most $2n$ to this sum.
Hence $\overline{e} = \overline{e}_i + \overline{e}_\partial \ge pn/2n =p/2$, and the estimate
in bullet~(1) follows from Lemma~\ref{reduced_graph_euler_count}.
\end{proof}

In the next proposition we consider the orientable Seifert fiber spaces $D^2(2)$, 
$M^2(0)$, $A^2(1)$ and $P^2(0)$, where $F^2(n)$ denotes a Seifert fiber 
space with base surface $F^2$ and $n$ exceptional fibers, and $D^2$, $M^2$, 
$A^2$, $P^2$ are the disk, M\"obius band, annulus, and pair of pants, 
respectively. These, together with $S^1 \times D^2$ and $T^2 \times I$, are precisely the
atoroidal Seifert fiber spaces with non-empty boundary.

\begin{remark}
This notation does not uniquely determine the space up to homeomorphism, since we do
not specify the type of singular fiber.
\end{remark}

\begin{proposition}\label{special_Seifert_proposition}
Let $M$ be an atoroidal Seifert fiber space of the kind $D^2(2)$, $M^2(0)$, 
$A^2(1)$ or $P^2(0)$, and let $K$ be a knot in $M$.
Let $F$ be a relative $p$-Seifert surface for $K$.
Then either 
\begin{enumerate}
\item{$-\chi^- (F)\ge p/6$; or}
\item{$K$ is isotopic into $\partial M$; or}
\item{$K$ is isotopic to a cable of an exceptional Seifert fiber.}
\end{enumerate}
\end{proposition}

\begin{proof}
We may assume that $F$ is good and $M-K$ is irreducible.

First assume that $F$ is not an annulus. 
Let $A$ be an essential annulus in $M$ such that $M$ cut along $A$ is two 
fibered solid tori if $M= D^2 (2)$, one fibered solid torus if 
$M= M^2(0)$ or $A^2(1)$, and $T^2 \times I$ if $M= P^2(0)$.
By Proposition~\ref{annular_proposition} either bullet~(1) or (2) holds, or $K$ can be isotoped 
to be disjoint from $A$.

If $M=P^2(0)$ and $K$ can be isotoped to be disjoint from $A$, then $K \subset T \times I$ where $T$
is a boundary component of $M$. Therefore bullet~(1) or (2) holds, by Corollary~\ref{torus_corollary}.

If $M$ is one of $D^2(2)$, $M^2(0)$ or $A^2(1)$ and $K$ can be isotoped to be disjoint from $A$,
then $K$ is contained in a fibered solid torus $V\subset M$. 
Let $F' = F\cap V$ and $F'' = F\cap (M- \inte V)$. By hypothesis $M-K$ is irreducible, so
$\partial V$ is incompressible in $M- \inte N(K)$. Therefore we can assume that 
$F''$ has no disk (or sphere) components. 
Hence $-\chi^- (F) \ge -\chi^- (F')$. 
Therefore by Proposition~\ref{compressible_torus_proposition} either $-\chi^- (F) >p/6$ or $K$ is 
a cable of $K_0$, the core of $V$. 
If $K_0$ is an ordinary fiber then it lies in a vertical incompressible 
torus $T$, so by Corollary~\ref{torus_corollary} either bullet~(1) holds or 
$K$ is isotopic to $K_0$, in which case bullet~(2) holds. 
If $K_0$ is an exceptional fiber we get conclusion (3).

Finally we consider the case where $F$ is an annulus, i.e. $\Gamma_F$ is 
a beachball.

If $\Gamma_F$ is a beachball of the first kind then (see Theorem~\ref{vanishing_theorem} (3)) 
$K$ is contained in a submanifold $N$ of $M$ of the form $M^2(0)$ or 
$M^2(1)$, as an ordinary or exceptional fiber respectively. 
Since $M$ is irreducible, atoroidal, and has non-empty boundary, the 
torus $\partial N$ is boundary parallel in $M$, and so $M\cong N$. 
Therefore $N$ is of the form $M^2(0)$, and $K$ is isotopic into $\partial M$.

If $\Gamma_F$ is a beachball of the second kind, then $M= V\cup_B W$, where 
$V$ is a solid torus, $B$ is an annulus with winding number $r\ge1$ in $V$,
and $K$ is a core of $V$.
If $r=1$ then $K$ is isotopic into $\partial M$. 
If $r>1$ then the form of  $M$ implies that the (separating) annulus $B$ 
is essential in $M$, and hence is vertical.
Therefore $K$ is an exceptional fiber in $M$.
\end{proof} 

We are finally ready to treat the case that $M$ is toroidal.

\begin{theorem}\label{toroidal_theorem}
Let $M$ be a closed, irreducible, toroidal 3-manifold, and let $K$ be a 
knot in $M$. 
Then either 
\begin{enumerate}
\item{$\|K\| \ge 1/402$; or}
\item{$K$ is trivial; or}
\item{$K$ is contained in a hyperbolic piece $N$ of the 
JSJ decomposition of $M$ and is isotopic either to a cable of a core of a 
Margulis tube or into a component of $\partial N$; or}
\item{$K$ is contained in a Seifert fiber piece $N$ of the 
JSJ decomposition of $M$ and either
\begin{itemize}
\item[(A)]{$K$ is isotopic to an ordinary fiber or a cable of an exceptional fiber or into $\partial N$, or}
\item[(B)]{$N$ contains a copy $Q$ of the twisted $S^1$ bundle over the M\"obius band and $K$ is contained
in $Q$ as a fiber in this bundle structure;}
\end{itemize}
or}\label{Seifert_case}
\item{$M$ is a $T^2$-bundle over $S^1$ with Anosov monodromy and $K$ is contained in a fiber.}\label{torus_case}
\end{enumerate}
\end{theorem}
\begin{remark}
If $K$ is disjoint from the hyperbolic pieces in the JSJ decomposition of $M$, the constant
$1/402$ can be improved to $1/24$.
\end{remark}

\begin{remark}
In Case~(4)(B), $\|K\|=0$ (see Theorem~\ref{vanishing_theorem} (3)). Also, the Seifert fibration of $Q$ induced from $N$
may be the one with base orbifold a disk with two cone points of order 2, in which case $K$ is not a
Seifert fiber in $N$.
\end{remark}

\begin{remark}
Strictly speaking, (5) is a special case of (4) where the Seifert fiber piece $N$ is $T^2 \times I$, but
for clarity we list it separately.
\end{remark}

\begin{proof}
We may assume that $M-K$ is irreducible. 
By Proposition~\ref{toroidal_atoroidal_proposition}, either $\|K\|\ge 1/24$
or $M-K$ is toroidal. So we may assume that $M-K$ is irreducible and toroidal.

Let $\T$ be a maximal disjoint union of non-parallel essential tori in $M-K$ (note that $\T$ is nonempty).
Let $N$ be the component of $M$ cut along $\T$ that contains $K$.
Then $N$ is irreducible, has boundary a non-empty disjoint union of tori, 
and $N-K$ is irreducible and atoroidal.

Let $S$ be a good $p$-Seifert surface for $K$ in $M$.
Then $F= S\cap N$ is a relative $p$-Seifert surface for $K$ in $N$.
Let $S_0 = \overline{S-F}$.
Since $\partial N$ is incompressible in $M-K$ we may assume that no 
component of $S_0$ or $F$ is a disk (or sphere). 
Hence $\eta (S) = \eta (F) + \eta (S_0)$ and so $\eta (S) \ge \eta (F)$.

If $N$ is toroidal then we are done by Proposition~\ref{toroidal_atoroidal_proposition}.
If $N$ is atoroidal then $N$ is either hyperbolic or Seifert fibered. 
In the former case, $N$ is a piece of the JSJ decomposition of $M$, and 
by Proposition~\ref{hyperbolic_proposition},
either (1) or (3) holds. 

Suppose $N$ is a Seifert fiber space. 
Since $N$ is atoroidal and has non-empty boundary it is either 
homeomorphic to $S^1\times D^2$ or $T^2 \times I$, or has Seifert fiber 
structure $D^2(2)$, $M^2 (0)$, $A^2(1)$ or $P^2(0)$. 
In the last four cases the result follows from Proposition~\ref{special_Seifert_proposition}
(to get conclusion (4) of the theorem we replace the $N$ considered here 
with the Seifert fiber piece of the JSJ decomposition of $M$ that 
contains it).

If $N = T^2\times I$ then by Proposition~\ref{torus_product_proposition} either (1) holds 
or $K$ is isotopic onto $T = T^2 \times \{1/2\}$. 
Since $T$ is an incompressible torus in $M$ there are four possibilities:
\begin{itemize}
\item[(i)] $T$ is a torus in the JSJ decomposition of $M$;
\item[(ii)] $T$ is a vertical essential torus in a Seifert fiber 
piece of the JSJ decomposition of $M$;
\item[(iii)] $M$ is a Seifert fiber space and $T$ is horizontal;
\item[(iv)] $M$ is a $T^2$-bundle over $S^1$ and $T$ is a fiber.
\end{itemize}

If (i) holds we are done. If (ii) holds we are done by Proposition~\ref{prop2.24}.

Suppose (iii) holds but not (iv). 
Then $T$ separates $M$ into two twisted $I$-bundles over the Klein bottle.
The Seifert fibering of $M$ must have Euler number~0, and $M$ is either 
a twisted $S^1$-bundle over the Klein bottle or has base orbifold 
$\RP^2$ with two orbifold points of order~2. 
In both cases $M$ has an isomorphic (but non-isotopic) Seifert fibering
in which $T$ is vertical.

If (iv) holds then by Remark~\ref{torus_bundle_remark}
either $\|K\|=0$, or $\|K\| \ge 1/8$, or we are in Case~G or Case~H of \S~\ref{torus_bundles_subsection}.
In Case~G, $M$ is Seifert fibered and $K$ is a fiber.
Case~H is conclusion (5). 
So suppose $\|K\|=0$.
Then by part~(2) of Theorem~\ref{torus_bundle_theorem}
$\left(\begin{smallmatrix} \alpha&\beta\\ \gamma&\delta\end{smallmatrix}
\right) = \left(\begin{smallmatrix} -1&n\\ 0&-1\end{smallmatrix}\right)$,
with $K$ representing the first element of the corresponding basis.
Thus $M$ also has the structure of an $S^1$-bundle over the Klein bottle
with Euler number $n$, where $K$ is a fiber. 

\smallskip

Finally, suppose $N = S^1\times D^2$. 
Then by Proposition~\ref{compressible_torus_proposition}, and the fact that $\T$ is essential in 
$M-K$, we may assume that $K$ is a non-trivial cable of $K_0$, the core of $N$.
By Proposition~\ref{satellite_proposition}, $\|K_0\| < \|K\|$. 

Now repeat the whole argument with $K_0$ in place of $K$. 
We conclude that either (1) holds or 
\begin{itemize}
\item[(a)]{$K_0$ is contained in a  hyperbolic piece $N_0$ of the 
JSJ decomposition of $M$ and is isotopic either to a cable of a core of 
a Margulis tube or into $\partial N_0$; or}
\item[(b)]{$K_0$ is contained in a Seifert fiber piece $N_0$ of the JSJ decomposition of $M$ and is
isotopic either to an ordinary fiber or a cable of an exceptional fiber of $N_0$ or into $\partial N_0$,
or $N_0$ contains a copy $Q_0$ of the twisted $S^1$ bundle over the M\"obius band and $K$ is
contained in $Q_0$ as an $S^1$ fiber; or}
\item[(c)]{$M$ is a $T^2$-bundle over $S^1$ and $K_0$ lies in a fiber; or}
\item[(d)]{$K_0$ is a non-trivial cable of a knot $K_1$.}
\end{itemize}

By Proposition~\ref{compressible_torus_proposition}, $K$ cannot be a non-trivial cable of a 
non-trivial cable. 
Similarly, by Proposition~\ref{torus_product_proposition}, $K$ cannot be a non-trivial cable 
of an essential curve in an incompressible torus in $M$. This completes the proof of the theorem.
\end{proof}

\end{document}